\documentclass[twoside,11pt]{article}
\usepackage{a4}
\usepackage{amssymb}
\usepackage{amsmath}
\usepackage{amsthm}
\usepackage{enumerate}
\usepackage{color}

\newtheorem{theorem}{Theorem}
\newtheorem{corollary}{Corollary}
\newtheorem{lemma}{Lemma}

\newtheorem*{claim*}{Claim}

\newtheorem{theorem MTP}{Mass Transference Principle}
\newtheorem*{theorem MTP*}{Mass Transference Principle}
\newtheorem{theorem K}{Khintchine's Theorem}
\newtheorem*{theorem K*}{Khintchine's Theorem}
\newtheorem{theorem J}{Jarn\'{\i}k's Theorem}
\newtheorem*{theorem J*}{Jarn\'{\i}k's Theorem}
\newtheorem{theorem KJ}{Khintchine--Jarn\'{\i}k Theorem}
\newtheorem*{theorem KJ*}{Khintchine--Jarn\'{\i}k Theorem}
\newtheorem{theorem BV1}{Theorem BV1}
\newtheorem*{theorem BV1*}{Theorem BV1}
\newtheorem{theorem BV2}{Theorem BV2}
\newtheorem*{theorem BV2*}{Theorem BV2}
\newtheorem{theorem KG}{Theorem KG}
\newtheorem*{theorem KG*}{Theorem KG}
\newtheorem{theorem IHKG}{Inhomogeneous Khintchine--Groshev Theorem}
\newtheorem*{theorem IHKG*}{Inhomogeneous Khintchine--Groshev Theorem}
\newtheorem{theorem DLN1}{Theorem DLN1}
\newtheorem*{theorem DLN1*}{Theorem DLN1}
\newtheorem{theorem DLN2}{Theorem DLN2}
\newtheorem*{theorem DLN2*}{Theorem DLN2}
\newtheorem{theorem DLN3}{Theorem DLN3}
\newtheorem*{theorem DLN3*}{Theorem DLN3}
\newtheorem{theorem S}{Theorem S}
\newtheorem*{theorem S*}{Theorem S}
\theoremstyle{definition}

\theoremstyle{remark}
\newtheorem{remark}{Remark}
\newtheorem*{remark*}{Remark}

\renewcommand{\Bbb}[1]{\mathbb{#1}}
\newcommand{\N}{{\Bbb N}}         
\newcommand{\I}{{\Bbb I}}
\newcommand{\R}{{\Bbb R}}        
\newcommand{\Z}{{\Bbb Z}}         
\newcommand{\K}{{\Bbb K}}         

\newcommand{\cA}{{\cal A}}

\newcommand{\cC}{{\cal C}}

\newcommand{\cG}{{\cal G}}
\newcommand{\cH}{{\cal H}}

\newcommand{\cK}{{\cal K}}

\newcommand{\cM}{{\cal M}}

\newcommand{\cR}{{\cal R}}

\newcommand{\ve}{\varepsilon}

\newcommand{\Om}{\Omega}
\newcommand{\U}{\Upsilon}
\newcommand{\La}{\Lambda}

\newcommand{\tU}{\tilde\Upsilon}

\newcommand{\Pj}{\Phi_j}

\newcommand{\diam}{r}
\newcommand{\dist}{\operatorname{dist}}

\newcommand{\vv}[1]{{\mathbf{#1}}}

\renewcommand{\le}{\leq}
\renewcommand{\ge}{\geq}

\newcommand{\kgb}{K_{G,B}}

\newcommand{\hf}{\cH^f}
\newcommand{\hs}{\cH^s}

\newcommand{\Lj}{(L;j)}
\newcommand{\aj}{(A;j)}
\newcommand{\ajj}{(A';j')}

\newcommand{\caj}{\cC\aj}

\newcommand{\x}{\mathbf{x}}
\newcommand{\y}{\mathbf{y}}
\newcommand{\p}{\mathbf{p}}
\newcommand{\q}{\mathbf{q}}
\newcommand{\bv}{\mathbf{v}}
\newcommand{\0}{\mathbf{0}}

\newcommand{\dimh}{\dim_H}

\begin{document}

\title{\bf A Mass Transference Principle for systems of linear forms and its applications}
\author{Demi Allen\footnote{Supported by EPSRC Doctoral Training Grant: EP/M506680/1} \and Victor Beresnevich\footnote{Supported by EPSRC Programme Grant: EP/J018260/1}}
\date{\footnotesize{\it Dedicated to Denise and John Allen --- Nan and Grandad --- \\
on the occasions of their 70th birthdays.}}
\maketitle
\begin{abstract}
In this paper we establish a general form of the Mass Transference Principle for systems of linear forms conjectured in \cite{BBDV ref}. We also present a number of applications of this result to problems in Diophantine approximation. These include a general transference of Lebesgue measure Khintchine--Groshev type theorems to Hausdorff measure statements. The statements we obtain are applicable in both the homogeneous and inhomogeneous settings as well as allowing transference under any additional constraints on approximating integer points. In particular, we establish Hausdorff measure counterparts of some Khintchine--Groshev type theorems with primitivity constraints recently proved by Dani, Laurent and Nogueira \cite{DLN ref}.
\end{abstract}
\noindent{\small 2000 {\it Mathematics Subject Classification}\/: Primary 11J83, 28A78; Secondary 11J13, 11K60


\section{Introduction}

The main goal of this paper is to settle a problem posed in \cite{BBDV ref} regarding the Mass Transference Principle, a technique in geometric measure theory that was originally discovered in \cite{BV MTP} having primarily been motivated by applications in Metric Number Theory. To some extent the present work is also driven by such applications.

To begin with, recall that the sets of interest in Metric Number Theory often arise as the upper limit of a sequence of `elementary' sets, such as balls, and satisfy elegant zero-one laws. Recall that if $(E_i)_{i \in \N}$ is a sequence of sets then the \emph{upper limit} or \emph{limsup} of this sequence is defined as 
\begin{align*}
\limsup_{i \to \infty}{E_i} :&= \{x: x \in E_i \text{ for infinitely many } i \in \N\} \\
                             &= \bigcap_{n \in \N}{\bigcup_{i \geq n}{E_i}}.
\end{align*}
These zero-one laws usually involve simple criteria, typically the convergence or divergence of a certain sum, for determining whether the measure of the $\limsup$ set is zero or one. To give an example, consider Khintchine's classical theorem \cite{kh} that deals with the set $\cK(\psi)$ of $x\in[0,1]$ such that
\begin{equation}\label{Khin}
|qx-p|<\psi(q)
\end{equation}
holds for infinitely many $(p,q)\in\Z\times\N$. Clearly, $\cK(\psi)$ is the $\limsup$ set of the intervals defined by \eqref{Khin}, which are the `elementary' sets in this setting. Khintchine proved that for any arithmetic function $\psi:\N\to\R^+:=[0,+\infty)$ such that $q\psi(q)$ is monotonic the Lebesgue measure of $\cK(\psi)$ is {\sc zero} if $\sum_{q=1}^\infty\psi(q)<\infty$ and {\sc one} otherwise.

Around 1930 Jarn\' ik and Besicovitch both independently considered the size of $\cK(\psi)$ using Hausdorff measures and dimension, thus proving results enabling us to see the difference between sets $\cK(\psi)$ indistinguishable by Khintchine's result. For example, the Jarn\' ik--Besicovitch Theorem says that the Hausdorff dimension of $\cK(q\mapsto q^{-v})$ is $\tfrac2{v+1}$ for $v>1$.

Over time the findings of Khintchine, Jarn\'ik and Besicovitch have been sharpened and generalised in numerous ways, including to involve problems concerning systems of linear forms. The theories for the ambient measure and Hausdorff measures had been evolving relatively separately until the discovery of the so-called Mass Transference Principle \cite{BV MTP}. This is a technique that enables one to easily obtain Hausdorff measure statements from a priori less general Lebesgue measure statements.

Let $f$ be a dimension function and let $\cH^f(\,\cdot\,)$ denote Hausdorff $f$-measure (see \S\ref{HM} for definitions). Given a ball $B := B(x,r)$ in $\R^k$ of radius $r$ centred at $x$, let $B^f := B(x,f(r)^{\frac{1}{k}})$. When $f(x) = x^s$ for some $s > 0$ we will denote $B^f$ by $B^s$. In particular, we always have  that $B^k = B$. The following statement is the main result of \cite{BV MTP}.

\begin{theorem MTP*}
Let $\{B_j\}_{j\in\N}$ be a sequence of balls in $\R^k$ with
$\diam(B_j)\to 0$ as $j\to\infty$. Let $f$ be a dimension function
such that $x^{-k}f(x)$ is monotonic. Suppose that for any ball $B$ in $\R^k$
\[ \cH^k\big(\/B\cap\limsup_{j\to\infty}B^f_j{}\,\big)=\cH^k(B) \ .\]
Then, for any ball $B$ in $\R^k$
\[ \cH^f\big(\/B\cap\limsup_{j\to\infty}B^k_j\,\big)=\cH^f(B) \ .\]
\end{theorem MTP*}

The original Mass Transference Principle \cite{BV MTP} stated above is a result regarding $\limsup$ sets which arise from sequences of balls. For the sake of completeness, we remark here that recently some progress has been made towards extending the Mass Transference Principle to deal with $\limsup$ sets defined by sequences of rectangles \cite{WWX ref}. In this paper we will be dealing with the extension of the Mass Transference Principle in the setting where we are interested in approximation by planes. This is not a new direction of research. Indeed, such an extension has already been obtained in \cite{BV Slicing}. However, the mass transference principle result of \cite{BV Slicing} carries some technical conditions which arise as a consequence of the ``slicing" technique that was used for the proof. These conditions were conjectured to be unnecessary and verifying that this is indeed the case is the main purpose of this paper.

Let $k, m \geq 1$ and $l \geq 0$ be integers such that $k = m + l$. Let $\cR := (R_j)_{j \in \N}$ be a family of planes in $\R^k$ of common dimension $l$. For every $j \in \N$ and $\delta \geq 0$, define
\[ \Delta(R_j,\delta) := \{\mathbf{x} \in \R^k : \dist(\mathbf{x},R_j)<\delta\},\]
where $\dist(\vv x,R_j)=\inf\{\|\vv x-\vv y\|:\vv y\in R_j\}$ and $\|\cdot\|$ is any fixed norm on $\R^k$.

Let $\U : \N \to \R : j \mapsto \U_j$ be a non-negative real-valued function on $\N$ such that $\U_j \to 0$ as $j \to \infty$.
Consider
\[ \La(\U) := \{\mathbf{x} \in \R^k : \mathbf{x} \in \Delta(R_j, \U_j) \text{ for infinitely many } j \in \N\}.\]

In \cite{BV Slicing}, the following was established.

\begin{theorem BV1*}
Let $\cR$ and $\U$ be as given above. Let $V$ be a linear subspace of $\R^k$ such that $\dim V = m  = \text{\textnormal{codim }} \cR$,
\begin{enumerate}
\item[{\rm(i)}]{$V \cap R_j \neq \emptyset$ for all $j \in \N$, and}
\item[{\rm(ii)}]{$\sup_{j \in \N}{\operatorname{diam}(V \cap \Delta(R_j,1))} < \infty$.}
\end{enumerate}
 Let $f$ and $g : r \to g(r) := r^{-l}f(r)$ be dimension functions such that $r^{-k}f(r)$ is monotonic and let $\Om$ be a ball in $\R^k$. Suppose that for any ball $B$ in $\Om$
\[\cH^k\left(B \cap \La\left(g(\U)^{\frac{1}{m}}\right)\right) = \cH^k(B).\]
Then, for any ball $B$ in $\Om$
\[\cH^f(B \cap \La(\U)) = \cH^f(B).\]
\end{theorem BV1*}

\begin{remark*}
In the case that $l = 0$ and $\Omega=\R^k$, Theorem BV1 coincides with the Mass Transference Principle stated above.
\end{remark*}

The conditions (i) and (ii) in Theorem BV1 arise as a consequence of the particular proof strategy employed in \cite{BV Slicing}. However, it was conjectured \cite[Conjecture E]{BBDV ref} that Theorem BV1 should be true without conditions (i) and (ii). By adopting a different proof strategy --- one similar to that used to prove the Mass Transference Principle in \cite{BV MTP} rather than ``slicing" --- we are able to remove conditions (i) and (ii) and, consequently, prove the following.

\begin{theorem} \label{mtp for linear forms theorem}
Let $\cR$ and $\U$ be as given above. Let $f$ and $g : r \to g(r) := r^{-l}f(r)$ be dimension functions such that $r^{-k}f(r)$ is monotonic and let $\Om$ be a ball in $\R^k$. Suppose that for any ball $B$ in $\Om$
\begin{equation}\label{w1}
\cH^k\left(B \cap \La\left(g(\U)^{\frac{1}{m}}\right)\right) = \cH^k(B).
\end{equation}
Then, for any ball $B$ in $\Om$
\begin{equation}\label{w2}
\cH^f(B \cap \La(\U)) = \cH^f(B).
\end{equation}
\end{theorem}

At first glance, conditions (i) and (ii) in Theorem BV1 do not seem particularly restrictive. Indeed, there are a number of interesting consequences of this theorem --- see \cite{BBDV ref, BV Slicing}. However, in the following section we present applications of Theorem~\ref{mtp for linear forms theorem} which may well be out of reach when using Theorem~BV1.
In Sections \ref{preliminaries section} and \ref{KBG lemma section} we establish  necessary preliminaries and some auxiliary lemmas before presenting the full proof of Theorem~\ref{mtp for linear forms theorem} in Section \ref{proof section}.

\section{Some applications of Theorem \ref{mtp for linear forms theorem}}

In this section we highlight merely a few applications of Theorem \ref{mtp for linear forms theorem} which we hope give an idea of the breadth of its consequences.
In \S\ref{KG section} we show that, using Theorem \ref{mtp for linear forms theorem}, with relative ease we are able to remove the last remaining monotonicity condition from a Hausdorff measure analogue of the classical Khintchine--Groshev theorem. We also show how the same outcome may be achieved, albeit with a somewhat longer proof, by using Theorem BV1 instead of Theorem \ref{mtp for linear forms theorem}. In \S\ref{inhomogeneous KG section} we obtain a Hausdorff measure analogue of the inhomogeneous version of the Khintchine--Groshev theorem.

In \S\ref{primitive approximation section} we present Hausdorff measure analogues of some recent results of Dani, Laurent and Nogueira \cite{DLN ref}. They have established Khintchine--Groshev type statements in which the approximating points $(\p,\q)$ are subject to certain primitivity conditions. We obtain the corresponding Hausdorff measure results.
On the way to realising some of the results outlined above, in \S\ref{inhomogeneous KG section} and \S\ref{primitive approximation section} we develop several more general statements which reformulate Theorem \ref{mtp for linear forms theorem} in terms of transferring Lebesgue measure statements to Hausdorff measure statements for very general sets of $\Psi$-approximable points (see Theorems \ref{general statement: psi depending on q}, \ref{general statement: psi depending on |q|} and \ref{general theorem: psi depending on (p,q)}).
The recurring theme throughout this section is that, given more-or-less any Khintchine--Groshev type statement, Theorem~\ref{mtp for linear forms theorem} can be used to establish the corresponding Hausdorff measure result.

\subsection{The Khintchine--Groshev Theorem for Hausdorff measures} \label{KG section}

Let $n \geq 1$ and $m \geq 1$ be integers. Denote by $\mathbb{I}^{nm}$ the unit cube $[0,1]^{nm}$ in $\R^{nm}$. Throughout this section we consider $\R^{nm}$ equipped with the norm $\|\cdot\| : \R^{nm} \to \R$ defined as follows
\begin{equation}\label{norm}
\|\x\| = \sqrt{n}\max_{1 \leq \ell \leq m}{|\x_\ell|_2}
\end{equation}
where $\x = (\x_1, \dots, \x_m)$ with each $\x_\ell$ representing a column vector in $\R^n$ for $1 \leq \ell \leq m$, and $|\cdot|_2$ is the usual Euclidean norm on $\R^n$. The role of the norm \eqref{norm} will become apparent soon, namely through the proof of Theorem~\ref{Khintchine--Groshev Hausdorff analogue theorem} below.

Given a function $\psi : \N \to \R^+$, let $\cA_{n,m}(\psi)$ denote the set of $\x \in \mathbb{I}^{nm}$ such that
\[|\q\x + \p| < \psi(|\q|)\]
for infinitely many $(\p,\q) \in \Z^m \times \Z^n \setminus \{\0\}$. Here, $|\cdot|$ denotes the supremum norm, $\x = (x_{i\ell})$ is regarded as an $n \times m$ matrix and $\q$ is regarded as a row vector. Thus, $\q\x$ represents a point in $\R^m$ given by the system
\[q_1x_{1\ell} + \dots + q_nx_{n\ell} \quad (1 \leq \ell \leq m)\]
of $m$ real linear forms in $n$ variables. We will say that the points in $\cA_{n,m}(\psi)$ are $\psi$-approximable. That $\cA_{n,m}(\psi)$ satisfies an elegant zero-one law in terms of $nm$-dimensional Lebesgue measure when the function $\psi$ is monotonic is the content of the classical Khintchine--Groshev Theorem. We opt to state here a modern version of this result which is best possible (see \cite{BV KG}).

In what follows $|X|$ will denote the $k$-dimensional Lebesgue measure of $X\subset\R^k$.

\begin{theorem BV2*}
Let $\psi : \N \to \R^+$ be an approximating function and let $nm > 1$. Then
$$
|\cA_{n,m}(\psi)| =\left\{
\begin{array}{lcl}
 \displaystyle 0 \quad & \mbox{ \text{if }
 $\sum_{q=1}^{\infty}{q^{n-1}\psi(q)^m} < \infty$,}\\[5ex]
 \displaystyle 1 \quad & \mbox{ \text{if }
 $\sum_{q=1}^{\infty}{q^{n-1}\psi(q)^m} = \infty$.}
\end{array}
\right.
$$
\end{theorem BV2*}

The earliest versions of this theorem were due to Khintchine and Groshev and included various extra constraints including monotonicity of $\psi$. A famous counterexample constructed by Duffin and Schaeffer \cite{Duffin-Schaeffer ref} shows that, while Theorem BV2 also holds when $m=n=1$ and $\psi$ is monotonic, the monotonicity condition cannot be removed when $m = n = 1$ and so it is natural to exclude this situation by letting $nm > 1$. In the latter case, the monotonicity condition has been removed completely, leaving Theorem BV2. That monotonicity may be removed in the case $n=1$ is due to a result of Gallagher and in the case where $n>2$ it is a consequence of a result due to Schmidt. For further details we refer the reader to \cite{BBDV ref} and references therein. The final unnecessary monotonicity condition to be removed was the $n = 2$ case. Formally stated as Conjecture A in \cite{BBDV ref}, this case was resolved in \cite{BV KG}.

Regarding the Hausdorff measure theory we shall show the following.

\begin{theorem} \label{Khintchine--Groshev Hausdorff analogue theorem}
Let $\psi: \N \to \R^+$ be any approximating function and let $nm > 1$. Let $f$ and $g: r \to g(r) := r^{-m(n-1)}f(r)$ be dimension functions such that $r^{-nm}f(r)$ is monotonic. Then,
$$
\cH^f(\cA_{n,m}(\psi)) =\left\{
\begin{array}{lcl}
 \displaystyle 0 \quad & \mbox{ \text{if }
 $\sum_{q=1}^{\infty}{q^{n+m-1}g\left(\frac{\psi(q)}{q}\right)} < \infty$,}\\[5ex]
 \displaystyle \cH^f(\mathbb{I}^{nm}) \quad & \mbox{ \text{if }
 $\sum_{q=1}^{\infty}{q^{n+m-1}g\left(\frac{\psi(q)}{q}\right)} = \infty$.}
\end{array}
\right.
$$
\end{theorem}

Theorem \ref{Khintchine--Groshev Hausdorff analogue theorem} is not entirely new and was in fact previously obtained in \cite{BBDV ref} via Theorem~BV1 subject to $\psi$ being monotonic in the case that $n = 2$. The deduction there was relying on a theorem of Sprind\v zuk rather than Theorem~BV2 (which is what we shall use). In fact, with several additional assumptions imposed on $\psi$ and $f$, the result was first obtained by Dickinson and Velani \cite{DV97}.
Indeed, the proof of the convergence case of Theorem~\ref{Khintchine--Groshev Hausdorff analogue theorem} makes use of standard covering arguments that, with little adjustment, can be drawn from \cite{DV97}.

In what follows we shall give two proofs for the divergence case of Theorem~\ref{Khintchine--Groshev Hausdorff analogue theorem}, one using Theorem~BV1 and the other using Theorem~\ref{mtp for linear forms theorem}. The reason for this is to show the advantage of using Theorem~\ref{mtp for linear forms theorem} on the one hand, and to explicitly exhibit obstacles in using Theorem~BV1 in other settings on the other hand. In the proofs we will use the following notation.
For $(\p,\q) \in \Z^m \times \Z^n \setminus \{\0\}$ let
\[R_{\p,\q} := \{\x \in \R^{nm}: \q\x + \p = \0\}.\]
Note that, throughout the proofs of Theorem~\ref{Khintchine--Groshev Hausdorff analogue theorem}, $(\p,\q)$ will play the role of the index $j$ appearing in Theorem~BV1 and Theorem~\ref{mtp for linear forms theorem}.
Also note that for $\delta \geq 0$ we have
\[\Delta(R_{\p,\q},\delta) = \{\x \in \R^{nm}: \dist(\x, R_{\p,\q}) < \delta\},\]
where
\[\dist(\x,R_{\p,\q}) = \inf_{\y \in R_{\p,\q}}{\|\x-\y\|} = \frac{\sqrt{n}|\q\x+\p|}{|\q|_2}.\]


We note that if $\psi(r)\geq 1$ for infinitely many $r\in\N$, then $\cA_{n,m}(\psi)=\I^{nm}$ and the divergence case of Theorem~\ref{Khintchine--Groshev Hausdorff analogue theorem} is trivial. Hence, without loss of generality we may assume that $\psi(r)\leq 1$ for all $r\in\N$.
First we show how
\begin{equation}\label{s1}
\text{\em Theorem~BV1 and Theorem~BV2 ~imply~ the divergence case of Theorem~\ref{Khintchine--Groshev Hausdorff analogue theorem}.}
\end{equation}
\begin{proof}
Recall that
\begin{equation}\label{div}
\sum_{q=1}^{\infty}{q^{n+m-1}g\left(\frac{\psi(q)}{q}\right)} = \infty.
\end{equation}
To use Theorem BV1 we have to restrict the approximating integer points $\vv q$ in order to meet conditions (i) and (ii) of Theorem~BV1. We will use the same idea as in \cite{BBDV ref}, namely we will impose the requirement that $|\vv q|=|q_K|$ for a fixed $K\in\{1,\dots,n\}$.  Sprind\v{z}uk's theorem that is used in \cite{BBDV ref} allows for the introduction of this requirement almost instantly. Unfortunately, this is not the case when one is using Theorem~BV2 and hence we will need a new argument. For each $1 \leq i \leq n$ define the auxiliary functions $\Psi_i: \Z^n \setminus \{\0\} \to \R^+$ by setting
\[\Psi_i(\q) = \left\{
\begin{array}{lcl}
 \displaystyle \psi(|\q|) \quad & \mbox{ \text{if }
 $|\q|=|q_i|$,}\\[3ex]
 \displaystyle 0 \quad & \mbox{ \text{otherwise}}.
\end{array}
\right.\]
In what follows, similarly to $\cA_{n,m}(\psi)$, we consider sets $\cA_{n,m}(\Psi)$ of points $\x \in \I^{nm}$ such that
$$
|\q\x + \p| < \Psi(\q)
$$
for infinitely many pairs $(\p,\q) \in \Z^m \times \Z^n \setminus \{\0\}$, where $\Psi: \Z^n \setminus \{\0\} \to \R^+$ is a multivariable function.
Since, by definition, $\Psi_i(\q) \leq \psi(|\q|)$ for each $1 \leq i \leq n$ and each $\vv q\in\Z^n\setminus\{\vv 0\}$, it follows that
\begin{equation}\label{f0}
\cA_{n,m}(\Psi_i) \subset \cA_{n,m}(\psi)\qquad\text{for each $1 \leq i \leq n$}.
\end{equation}
By \eqref{f0}, to complete the proof of \eqref{s1}, it is sufficient to show that
\begin{equation}\label{f2}
\cH^f(\cA_{n,m}(\Psi_K)) = \cH^f(\I^{nm})\qquad\text{for some $1\leq K\leq n$.}
\end{equation}
Without loss of generality we will assume that $K=1$. Define
\[S:= \{(\p,\q) \in \Z^m \times \Z^n \setminus \{\0\}: |\q| = |q_1| \text{ and } |\p|\leq M|\q|\},\]
where
\begin{equation}\label{M}
M=\max\Big\{2n,\sup_{r\in \N}\,\frac{2}{\sqrt{n}}\,g\left(\frac{\psi(r)}{r}\right)^{\frac{1}{m}}\Big\}\,.
\end{equation}
Note that since $g$ is increasing and $\psi(r)\le1$, the constant $M$ is finite.
Let
$\U_{\p,\q} := \dfrac{\Psi_1(\q)}{|\q|}$ for each $(\p,\q) \in S$. The purpose for introducing this auxiliary set $S$ will become apparent later.
Now, for each $(\p,\q) \in S$, 
\begin{align*}
\Delta(R_{\p,\q},\U_{\p,\q})\cap\I^{nm} &= \left\{\x \in \I^{nm}: \frac{\sqrt{n}|\q\x+\p|}{|\q|_2} < \frac{\Psi_1(\q)}{|\q|} \right\}\\[2ex]
& = \left\{\x \in \I^{nm}: |\q\x+\p| < \frac{|\q|_2\Psi_1(\q)}{\sqrt{n}|\q|} \right\}\\[2ex]
& \subset \Big\{\x \in \I^{nm}: |\q\x + \p| < \Psi_1(\q)\Big\}\,,
\end{align*}
since $|\q|_2\leq \sqrt{n}\,|\q|$.
It follows that
$\La(\U) \cap \I^{nm} \subset \cA_{n,m}(\Psi_1) \subset \I^{nm}$, where
$$
\La(\U) =\limsup_{(\p,\q) \in S}\Delta(R_{\p,\q},\U_{\p,\q})\,,
$$
and, in taking this limit, $(\p,\q)\in S$ can be arranged in any order.
Therefore, \eqref{f2} will follow on showing that
\begin{equation}\label{f4}
\cH^f(\La(\U) \cap \I^{nm} ) = \cH^f(\I^{nm})\,.
\end{equation}
Showing \eqref{f4} will rely on Theorem~BV1. First of all observe that conditions (i) and (ii) are met with
the $m$-dimensional subspace
\[V := \{\x = (\x_1, \x_2, \dots, \x_m) \in \R^{nm}: x_{i\ell} = 0 \text{ for all } \ell=1,\dots,m \text{ and } i=2,\dots,n\}.\]
Indeed, regarding condition (i), we have that $R_{\p,\q}\cap V$ consists of the single element
$$
\left( \begin{array}{cccc}
-\frac{p_1}{q_1}   & -\frac{p_2}{q_1}    &\dots     & -\frac{p_m}{q_1}   \\
0                  & 0                   & \dots    & 0       \\
\vdots             & \vdots              & \cdots   & \vdots  \\
0                  & 0                   & \dots    & 0       \end{array} \right)\,,
$$
and so is non-empty.
Regarding condition (ii),
for $(\p,\q) \in S$ we have that
\begin{align*}
V\cap\Delta(R_{\p,\q},1) & = \{\x \in V: \dist(\x,R_{\p,\q}) < 1\} \\[3ex]
               & = \left\{\x \in V: \frac{\sqrt{n}|\q\x+\p|}{|\q|_2}<1\right\}\\[3ex]
& = \left\{\x \in \R^{nm}:
\displaystyle\max_{1\leq \ell\leq m}\frac{\sqrt{n}|q_1 x_{1,\ell}+p_\ell|}{|\q|_2}<1\quad\text{and}\quad
x_{i\ell} = 0 \text{ for}~i\not=1
\right\}\\[3ex]
& \subset \left\{\x \in \R^{nm}:
\displaystyle\max_{1\leq \ell\leq m}\left|x_{1,\ell}+\frac{p_\ell}{q_1}\right|<1\quad\text{and}\quad x_{i\ell} = 0 \text{ for}~i\not=1
\right\}
\end{align*}
since $|q_1|=|\vv q|$ and $|\vv q|_2\leq\sqrt n|\vv q|$. Hence $\operatorname{diam}(V \cap \Delta(R_{\p,\q},1))\leq 2\sqrt{n}$ and we are done.

Now let
$\theta: \N \to \R^+$ be given by
$$
\theta(r) = \frac{r}{\sqrt{n}}g\left(\frac{\psi(r)}{r}\right)^{\frac{1}{m}}
$$
and, for each $1 \leq i \leq n$, let $\Theta_i: \Z^n \setminus \{\0\} \to \R^+$ be given by
$$
\Theta_i(\q) = \frac{|\q|}{\sqrt{n}}\,\,g\left(\frac{\Psi_i(\q)}{|\q|}\right)^{\frac{1}{m}}
~~=~~\left\{
\begin{array}{lcl}
 \displaystyle \theta(|\q|) \quad & \mbox{ \text{if }
 $|\q|=|q_i|$,}\\[3ex]
 \displaystyle 0 \quad & \mbox{ \text{otherwise}}.
\end{array}
\right.
$$
Similarly to \eqref{f0}, we have that $\cA_{n,m}(\Theta_i)\subset \cA_{n,m}(\theta)$ for each $1 \leq i \leq n$. Furthermore,
\begin{equation}\label{f1}
 \cA_{n,m}(\theta) = \bigcup_{i=1}^{n}{\cA_{n,m}(\Theta_i)}.
\end{equation}
Indeed, the `$\supseteq$' inclusion follows from the above. To show the converse, note that for any $\x \in \cA_{n,m}(\theta)$ the inequality $|\q\x + \p| < \theta(|\q|)$ is satisfied for infinitely many  $(\p, \q) \in \Z^m \times \Z^n \setminus \{\0\}$. Clearly, for each $\vv q\in\Z^n\setminus\{\vv 0\}$ we have that $\theta(|\q|)=\Theta_i(\vv q)$ for some $1\leq i\leq n$. Therefore, there is a fixed $i\in\{1,\dots,n\}$ such that $|\q\x + \p| < \theta(|\q|)=\Theta_i(\vv q)$ is satisfied for infinitely many  $(\p, \q) \in \Z^m \times \Z^n \setminus \{\0\}$. This means that $\vv x\in {\cA_{n,m}(\Theta_i)}$ for some $i$, thus verifying \eqref{f1}.

Next, observe that, by \eqref{div}, the sum
\[\sum_{q=1}^{\infty}{q^{n-1}\theta(q)^m} = \sum_{q=1}^{\infty}{\frac{q^{n+m-1}}{\sqrt{n}^m}g\left(\frac{\psi(q)}{q}\right)} = \frac{1}{\sqrt{n}^m}\sum_{q=1}^{\infty}{q^{n+m-1}g\left(\frac{\psi(q)}{q}\right)}\]
diverges. Therefore, by Theorem BV2, we have that $|\cA_{n,m}(\theta)| = 1$.
Hence, by \eqref{f1},
there exists some $1 \leq K \leq n$ such that $|\cA_{n,m}(\Theta_K)| > 0$. By the zero-one law
of \cite[Theorem 1]{BV Zero-one}, we know that $|\cA_{n,m}(\Theta_K)| \in \{0,1\}$. Hence,
\begin{equation}\label{f3}
|\cA_{n,m}(\Theta_K)| = 1.
\end{equation}
Without loss of generality we will suppose that $K = 1$, the same as in \eqref{f2}.

Now, using the fact that $|\q| \leq |\q|_2$, for $(\vv p,\vv q)\in S$ we have that
\begin{align*}
\Delta(R_{\p,\q},g(\U_{\p,\q})^{\frac{1}{m}})\cap \I^{nm}
            &=\left\{\x \in \I^{nm}: \frac{\sqrt{n}|\q\x+\p|}{|\q|_2} < g\left(\frac{\Psi_1(\q)}{|\q|}\right)^{\frac{1}{m}}\right\} \\
            &=\left\{\x \in \I^{nm}: |\q\x+\p| < \frac{|\q|_2}{\sqrt{n}}g\left(\frac{\Psi_1(\q)}{|\q|}\right)^{\frac{1}{m}}\right\} \\
            &\supseteq \left\{\x \in \I^{nm}: |\q\x+\p| < \frac{|\q|}{\sqrt{n}}g\left(\frac{\Psi_1(\q)}{|\q|}\right)^{\frac{1}{m}}\right\} \\
            &= \{\x \in \I^{nm}: |\q\x+\p|<\Theta_1(\q)\}.
\end{align*}
Furthermore observe that if $\{\x \in \I^{nm}: |\q\x + \p| < \Theta_1(\q)\} \not=\emptyset$, then
$|\p| \leq M|\q|$ and so $(\p,\q) \in S$.
Therefore,
\[\cA_{n,m}(\Theta_1) \subset \La(g(\U)^{\frac{1}{m}}) \cap \I^{nm} \subset \I^{nm}.\] In particular, $|\La(g(\U)^{\frac{1}{m}}) \cap \I^{nm}| = 1$ and so for any ball $B\subset \I^{nm}$ we have that $\cH^{nm}(\La(g(\U)^{\frac{1}{m}}) \cap B)=\cH^{nm}(B)$. Hence, we may apply Theorem~BV1 with $k = nm$, $l = m(n-1)$ and $m$ to conclude that, for any ball $B \subset \mathbb{I}^{nm}$, we have $\cH^f(B \cap \La(\U)) = \cH^f(B)$. In particular, $\cH^f(\La(\U)\cap\I^{nm}) = \cH^f(\I^{nm})$ and the proof is thus complete.
\end{proof}

We now show how
\begin{equation}\label{s2}
\text{\em Theorem~\ref{mtp for linear forms theorem} and Theorem~BV2 ~imply~ the divergence case of Theorem~\ref{Khintchine--Groshev Hausdorff analogue theorem}.}
\end{equation}

\begin{proof}
As before, we are given the divergence condition \eqref{div}.
For each pair $(\p,\q) \in \Z^m \times \Z^n \setminus \{\0\}$ with $|\p| \leq M|\q|$, where $M$ is given by \eqref{M}, let
$$R_{\p,\q} := \{\x \in \R^{nm} : \q\x + \p = \0\} \quad \text{and} \quad \U_{\p,\q} := \frac{\psi(|\q|)}{|\q|}.
$$
For such pairs $(\p,\q)$ we have that
\begin{align*}
\Delta(R_{\p,\q},\U_{\p,\q})\cap \I^{nm}  &= \left\{\x \in \I^{nm}: \frac{\sqrt{n}|\q\x+\p|}{|\q|_2} < \frac{\psi(|\q|)}{|\q|} \right\} \\[2ex]
&\subset \left\{\x \in \I^{nm}: |\q\x+\p| < \psi(|\q|)\right\}
\end{align*}
since $|\q|_2\leq \sqrt{n} \,|\q|$.
Therefore
\[\La(\U) \cap \I^{nm} \subset \cA_{n,m}(\psi) \subset \I^{nm},\]
where the $\limsup$ is taken over
$(\p,\q) \in \Z^m \times \Z^n \setminus \{\0\}$ with $|\vv p|\le M|\vv q|$.

Therefore, if we could show that $\cH^f(\La(\U) \cap \I^{nm}) = \cH^f(\I^{nm})$ the divergence part of Theorem~\ref{Khintchine--Groshev Hausdorff analogue theorem} would follow. 

Define $\theta: \N \to \R^+$ by 
\[\theta(r) = \frac{r}{\sqrt{n}}g\left(\frac{\psi(r)}{r}\right)^{\frac{1}{m}}\] 
and note that
\begin{align*}
\Delta(R_{\p,\q},g(\U_{\p,\q})^{\frac{1}{m}})\cap\I^{nm}
                               &= \left\{\x \in \I^{nm}: \frac{\sqrt{n}|\q\x+\p|}{|\q|_2} < g\left(\frac{\psi(|\q|)}{|\q|}\right)^{\frac{1}{m}} \right\} \\
                               &= \left\{\x \in \I^{nm}: |\q\x+\p| < \frac{|\q|_2}{\sqrt{n}}g\left(\frac{\psi(|\q|)}{|\q|}\right)^{\frac{1}{m}} \right\} \\
                               &\supseteq \left\{\x \in \I^{nm}: |\q\x+\p| < \frac{|\q|}{\sqrt{n}}g\left(\frac{\psi(|\q|)}{|\q|}\right)^{\frac{1}{m}} \right\} \\
                               &= \left\{\x \in \I^{nm}: |\q\x+\p| < \theta(|\q|) \right\}, 
\end{align*}
where this penultimate inclusion follows since $|\q| < |\q|_2$. Observe that if $\{\x \in \I^{nm}: |\q\x + \p| < \theta(|\q|)\} \not=\emptyset$, then
$|\p| \leq M|\q|$. It follows that
\[\cA_{n,m}(\theta) \subset \La(g(\U)^{\frac{1}{m}}) \cap \I^{nm}.\]

Now, by Theorem BV2 and the divergence condition (\ref{div}), we know that $|\cA_{n,m}(\theta)| = 1$ since
\[\sum_{q=1}^{\infty}{q^{n-1}\theta(q)^m} = \sum_{q=1}^{\infty}{\frac{q^{n+m-1}}{\sqrt{n}^m}g\left(\frac{\psi(q)}{q}\right)} = \infty.\]
Hence $|\La(g(\U)^{\frac{1}{m}}) \cap \I^{nm}| = 1$ and so we may apply Theorem \ref{mtp for linear forms theorem} with $k = nm$, $l = m(n-1)$ and $m$ to conclude that, for any ball $B \subset \mathbb{I}^{nm}$, we have $\cH^f(B \cap \La(\U)) = \cH^f(B)$. In particular, $\cH^f(\La(\U)\cap\I^{nm}) = \cH^f(\I^{nm})$ and the proof is thus complete.
%
\end{proof}

\medskip

\begin{remark}
Note that the proof of \eqref{s2} is not only shorter and simpler than that of \eqref{s1} but it also does not rely on the zero one law \cite[Theorem 1]{BV Zero-one}. This seemingly minor point becomes a substantial obstacle in trying to use the same line of argument as for \eqref{s1} in other settings, for example, in inhomogeneous problems. The point is that, as of now, we do not have an inhomogeneous
zero-one law similar to \cite[Theorem~1]{BV Zero-one} --- see \cite{Felipe} for partial results and further comments. The approach based on using Theorem~\ref{mtp for linear forms theorem}, on the other hand, works with ease in the inhomogeneous and other settings.
\end{remark}

\subsection{Inhomogeneous systems of linear forms} \label{inhomogeneous KG section}

In this section we will be concerned with the inhomogeneous version of the Khintchine--Groshev Theorem presented in the previous subsection.
Given an approximating function $\Psi: \Z^n \setminus \{\0\} \to \R^+$ and a fixed $\y \in \I^m$, we denote by $\cA_{n,m}^{\y}(\Psi)$ the set of $\x \in \I^{nm}$ for which
\[|\q\x + \p - \y| < \Psi(\q)\]
holds for infinitely many $(\p, \q) \in \Z^m \times \Z^n \setminus \{\0\}$.
In the case that $\Psi(\vv q)=\psi(|\vv q|)$ for some function $\psi: \N \to \R^+$ we write $\cA_{n,m}^{\y}(\psi)$ for $\cA_{n,m}^{\y}(\Psi)$.

%

Regarding inhomogeneous Diophantine approximation we have the following statement which can be deduced as a corollary of \cite[Chapter~1, Theorem~15]{Sprindzuk ref}. In the case that $\psi$ is monotonic this statement also follows as a consequence of the ubiquity technique, see \cite[\S12.1]{BDV limsup sets}.


\begin{theorem IHKG*}
Let $m,n \geq 1$ be integers and let $\y \in \I^m$. If $\psi: \N \to \R^+$ is an approximating function which is assumed to be monotonic if $n = 1$ or $n=2$, then
$$
|\cA_{n,m}^{\y}(\psi)| =\left\{
\begin{array}{lcl}
 \displaystyle 0 \quad & \mbox{ \text{if }
 $\sum_{q=1}^{\infty}{q^{n-1}\psi(q)^m} < \infty$,}\\[5ex]
 \displaystyle 1 \quad & \mbox{ \text{if }
 $\sum_{q=1}^{\infty}{q^{n-1}\psi(q)^m} = \infty$.}
\end{array}
\right.
$$
\end{theorem IHKG*}


The following is the Hausdorff measure version of the above statement.

\begin{theorem} \label{Inhomogeneous Khintchine--Groshev Hausdorff analogue theorem}
Let $m,n \geq 1$ be integers, let $\y \in \I^{m}$, and let $\psi: \N \to \R^+$ be an approximating function. Let $f$ and $g: r \to g(r):=r^{-m(n-1)}f(r)$ be dimension functions such that $r^{-nm}f(r)$ is monotonic. In the case that $n = 1$ or $n=2$ suppose also that $\psi$ is monotonically decreasing. Then,
$$
\cH^f(\cA_{n,m}^{\y}(\psi)) =\left\{
\begin{array}{lcl}
 \displaystyle 0 \quad & \mbox{ \text{if }
 $\sum_{q=1}^{\infty}{q^{n+m-1}g\left(\frac{\psi(q)}{q}\right)} < \infty$,}\\[5ex]
 \displaystyle \cH^f(\mathbb{I}^{nm}) \quad & \mbox{ \text{if }
 $\sum_{q=1}^{\infty}{q^{n+m-1}g\left(\frac{\psi(q)}{q}\right)} = \infty$.}
\end{array}
\right.
$$
\end{theorem}

The proof of the convergence case of Theorem~\ref{Inhomogeneous Khintchine--Groshev Hausdorff analogue theorem} once again makes use of standard covering arguments. The divergence case is a consequence of  the Inhomogeneous Khintchine--Groshev Theorem and Theorem \ref{mtp for linear forms theorem}. The proof of the divergence case is almost identical to that of \eqref{s2} and we therefore leave the details out. Furthermore, exploiting this same argument a little further, we can use Theorem \ref{mtp for linear forms theorem} to prove the following two more general statements from which both Theorems \ref{Khintchine--Groshev Hausdorff analogue theorem} and \ref{Inhomogeneous Khintchine--Groshev Hausdorff analogue theorem} follow as corollaries. In some sense Theorems~\ref{general statement: psi depending on q} and \ref{general statement: psi depending on |q|} below are reformulations of Theorem \ref{mtp for linear forms theorem} in terms of sets of $\Psi$-approximable (and $\psi$-approximable) points.

\medskip

\begin{theorem} \label{general statement: psi depending on q}
Let $\Psi: \Z^n \setminus \{\0\} \to \R^+$ be an approximating function and let $\y \in \I^m$.
Let $f$ and $g: r \to g(r):=r^{-m(n-1)}f(r)$ be dimension functions such that $r^{-nm}f(r)$ is monotonic. Let $$
\Theta: \Z^n \setminus \{\0\} \to \R^+\qquad\text{be defined by}\qquad \Theta(\q) = |\vv q|\,g\left(\frac{\Psi(\q)}{|\q|}\right)^{\frac{1}{m}}\,.
$$
Then
$$
|\cA_{n,m}^{\y}(\Theta)| = 1\qquad\text{implies}\qquad \cH^f(\cA_{n,m}^{\y}(\Psi)) = \cH^f(\I^{nm}).
$$
\end{theorem}

The following statement is a special case of Theorem~\ref{general statement: psi depending on q} with
$\Psi(\q) := \psi(|\q|)$.

\begin{theorem} \label{general statement: psi depending on |q|}
Let $\psi: \N \to \R^+$ be an approximating function, let $\y \in \I^m$ and let $f$ and $g: r \to g(r) := r^{-m(n-1)}f(r)$ be dimension functions such that $r^{-nm}f(r)$ is monotonic. Let
$$
\theta: \N \to \R^+\qquad\text{be defined by}\qquad \theta(r)= r\,g\left(\frac{\psi(r)}{r}\right)^{\frac{1}{m}}\,.
$$
Then
$$
|\cA_{n,m}^{\y}(\theta)| = 1\qquad\text{implies}\qquad \cH^f(\cA_{n,m}^{\y}(\psi)) = \cH^f(\I^{nm}).
$$
\end{theorem}

The proof of Theorem~\ref{general statement: psi depending on q} is similar to that of \eqref{s2}. We shall explicitly deduce it from the even more general result of \S\ref{primitive approximation section}, where the approximating function will be allowed to depend on $\vv p$ as well as $\vv q$.
Theorem~\ref{Inhomogeneous Khintchine--Groshev Hausdorff analogue theorem} now trivially follows on combining
the Inhomogeneous Khintchine--Groshev Theorem with Theorem~\ref{general statement: psi depending on |q|}. Furthermore, any progress in removing the monotonicity constraint on $\psi$ from the Inhomogeneous Khintchine--Groshev Theorem can be instantly transferred into a Hausdorff measure statement upon applying Theorem~\ref{general statement: psi depending on |q|}. Indeed, we suspect that a full inhomogeneous analogue of Theorem~BV2 must be true. Recall that it is open only in the case when $n=1$ or $n=2$.

\subsection{Approximation by primitive points and more} \label{primitive approximation section}

The key goal of this section is to present Hausdorff measure analogues of some recent results obtained by Dani, Laurent and Nogueira in \cite{DLN ref}. The setup they consider assumes certain coprimality conditions on the $(m+n)$-tuple $(q_1,\dots,q_n,p_1,\dots,p_m)$ of approximating integers. To achieve our goal we will first prove a very general statement which further extends Theorems \ref{general statement: psi depending on q} and \ref{general statement: psi depending on |q|} and is of independent interest. In particular, we will allow for the approximating function to depend on $(\p,\q)$ and will also introduce a `distortion' parameter $\Phi$ that allows certain flexibility within our framework. This allows us, for example, to incorporate the so-called `absolute value theory' \cite{D93,HK13,HL13}.

Within this section $\Psi: \Z^m \times \Z^n \setminus \{\0\} \to \R^+$ will be a function of $(\p,\q)$, $\y \in \I^m$ will be a fixed point and $\Phi \in \I^{mm}$ will be a fixed $m\times m$ square matrix. Further, define $\cM_{n,m}^{\y,\Phi}(\Psi)$ to be the set of $\x \in \I^{nm}$ such that
\[|\q\x + \p\Phi - \y| < \Psi(\p,\q)\]
holds for $(\p,\q) \in \Z^m\times\Z^n \setminus \{\0\}$ with arbitrarily large $|\q|$. Based upon Theorem~\ref{mtp for linear forms theorem} we now state and prove the following generalisation of Theorems~\ref{general statement: psi depending on q} and \ref{general statement: psi depending on |q|}.

\begin{theorem} \label{general theorem: psi depending on (p,q)}
Let $\Psi: \Z^m \times \Z^n \setminus \{\0\} \to \R^+$ be such that
\begin{align} \label{general theorem monotonicity condition}
\lim_{|\q| \to \infty}~\sup_{\p\in\Z^m}{\frac{\Psi(\p,\q)}{|\q|}} = 0\,,
\end{align}
and let $\y \in \I^m$ and $\Phi \in \I^{mm} \setminus \{\0\}$ be fixed. Let $f$ and $g: r \to g(r):=r^{-m(n-1)}f(r)$ be dimension functions such that $r^{-nm}f(r)$ is monotonic. Let
$$
\Theta: \Z^m \times \Z^n \setminus \{\0\} \to \R^+\qquad\text{be defined by}\qquad \Theta(\p,\q) = |\q|\,g\left(\frac{\Psi(\p,\q)}{|\q|}\right)^{\frac{1}{m}}.
$$
Then
$$
|\cM_{n,m}^{\y, \Phi}(\Theta)| = 1\qquad\text{implies}\qquad\cH^f(\cM_{n,m}^{\y,\Phi}(\Psi)) = \cH^f(\I^{nm})\,.
$$
\end{theorem}

\medskip

\begin{proof}
Let
\begin{equation}\label{MM}
M:=\max\Big\{3n,\sup_{(\p,\q) \in \Z^m \times \Z^n \setminus \{\0\}}\,\frac{3\Theta(\vv p,\vv q)}{\sqrt{n}|\vv q|}\Big\}\,.
\end{equation}
By the monotonicity of $g$ and condition \eqref{general theorem monotonicity condition}, we have that $M$ is finite. Let
$$
S:=\{(\p,\q)\in\Z^m\times \Z^n \setminus \{\0\}:|\vv p\Phi|\le M|\vv q|\}
$$
and let $S_\Phi$ be any fixed subset of $S$ such that for each $(\p',\q)\in S$ there exists $(\p,\q)\in S_\Phi$ such that
\begin{equation}\label{v}
\p\Phi=\p'\Phi\qquad\text{ and }\qquad\Theta(\p',\q)\le 2\Theta(\p,\q).
\end{equation}
Furthermore, let $S_{\Phi}$ be such that for all $(\p,\q),(\mathbf{r},\mathbf{s}) \in S_{\Phi}$ we have
\[(\p\Phi,\q) \neq (\mathbf{r}\Phi,\mathbf{s}) \qquad \text{if} \qquad (\p,\q) \neq (\mathbf{r},\mathbf{s}).\]
The existence of $S_{\Phi}$ is easily seen. For each $(\p,\q) \in S_\Phi$, let
\[
R_{\p,\q} := \{\x \in \R^{nm}: \q\x + \p\Phi - \y = 0\}\qquad\text{and}\qquad
\U_{\p,\q} := \frac{\Psi(\p,\q)}{|\q|}\,.
\]
For $(\p,\q)\in S_\Phi$ we have that
\begin{align*}
\Delta(R_{\p,\q},\U_{\p,\q})\cap \I^{nm}  &= \left\{\x \in \I^{nm}: \frac{\sqrt{n}|\q\x + \p\Phi - \y|}{|\q|_2} < \frac{\Psi(\p,\q)}{|\q|} \right\} \\[2ex]
&\subset \left\{\x \in \I^{nm}: |\q\x+\p\Phi-\vv y| < \Psi(\p,\q)\right\}
\end{align*}
since $|\q|_2\leq \sqrt{n} \,|\q|$. Also note that for each $\q\in\Z^n\setminus\{\vv0\}$ there are only finitely many $\p\in\Z^m$ such that $(\p,\q)\in S_\Phi$.
Therefore
\begin{equation}\label{incl}
\La(\U) \cap \I^{nm} \subset \cM_{n,m}^{\y,\Phi}(\Psi) \subset \I^{nm},
\end{equation}
where, when defining $\La(\U)$, the $\limsup$ is taken over $(\p,\q) \in S_\Phi$.
Hence, by \eqref{incl}, it would suffice for us to show that
\[\cH^f(\La(\U) \cap \I^{nm}) = \cH^f(\I^{nm}).\]
Consider $\La(g(\U)^{\frac{1}{m}})$, where the $\limsup$ is again taken over $(\p,\q) \in S_{\Phi}$. Take any $(\p',\q)\in S$ and let $(\p,\q)\in S_\Phi$ satisfy \eqref{v}.
Then, since $|\q| \leq |\q|_2$, we have that
\begin{align*}
\Delta(R_{\p,\q},g(\U_{\p,\q})^{\frac{1}{m}})\cap\I^{nm}
                   &= \left\{\x \in \I^{nm}: \frac{\sqrt{n}|\q\x + \p\Phi - \y|}{|\q|_2} < g\left(\frac{\Psi(\p,\q)}{|\q|}\right)^{\frac{1}{m}} \right\} \\
                   &= \left\{\x \in \I^{nm}: |\q\x + \p\Phi - \y| < \frac{|\q|_2}{\sqrt{n}}g\left(\frac{\Psi(\p,\q)}{|\q|}\right)^{\frac{1}{m}} \right\} \\
                   &\supseteq \left\{\x \in \I^{nm}: |\q\x + \p\Phi - \y| < \frac{|\q|}{\sqrt{n}} g\left(\frac{\Psi(\p,\q)}{|\q|}\right)^{\frac{1}{m}}\right\} \\
                   &= \left\{\x \in \I^{nm}: |\q\x + \p\Phi - \y| < \tfrac{1}{\sqrt n}\Theta(\p,\q)\right\}\\
                   &\supseteq \left\{\x \in \I^{nm}: |\q\x + \p'\Phi - \y| < \tfrac{1}{2\sqrt n}\Theta(\p',\q)\right\}.
\end{align*}
Also observe that if $\left\{\x \in \I^{nm}: |\q\x + \p'\Phi - \y| < \tfrac{1}{2\sqrt n}\Theta(\p',\q)\right\}\not=\emptyset$, then
$|\p'\Phi| \leq M|\q|$. It follows that
\begin{equation}\label{vv}
\cM_{n,m}^{\y,\Phi}(\tfrac{1}{2\sqrt n}\Theta) \subset \La(g(\U)^{\frac{1}{m}}) \subset \I^{nm}.
\end{equation}
Recall, that $|\cM_{n,m}^{\y,\Phi}(\Theta)| = 1$. Furthermore, in view of \cite[Lemma~4]{BV Zero-one}, we have that
$|\cM_{n,m}^{\y,\Phi}(\tfrac1{2\sqrt n}\Theta)| = 1$. Together with \eqref{vv} this implies that
$|\La(g(\U)^{\frac{1}{m}}) \cap \I^{nm}| = 1$. Further, note that, by
\eqref{general theorem monotonicity condition}, $\U_{\p,\q}\to 0$ as $|\q|\to\infty$. Therefore, Theorem~\ref{mtp for linear forms theorem} is applicable with $k = nm$, $l = m(n-1)$ and $m$ and we conclude that for any ball $B \subset \I^{nm}$ we have that $\cH^f(B \cap \La(\U)) = \cH^f(B)$. In particular, this means that $\cH^f(\I^{nm} \cap \La(\U)) = \cH^f(\I^{nm})$, as required.
\end{proof}

\begin{proof}[Proof of Theorem~\ref{general statement: psi depending on q}]
Let $\Psi$ be as in Theorem~\ref{general statement: psi depending on q}. First observe that if $\Psi(\q)\geq 1$ for infinitely many $\q\in\Z^n$, then $\cA_{n,m}^{\y}(\Psi)=\I^{nm}$ and there is nothing to prove. Otherwise we obviously have that $\Psi(\vv q)/|\vv q|\to0$ as $|\vv q|\to\infty$. In this case extending $\Psi$ to be a function of $(\p,\q)$ so that $\Psi(\p,\q):=\Psi(\q)$ and applying Theorem~\ref{general theorem: psi depending on (p,q)} we immediately recover Theorem~\ref{general statement: psi depending on q} from Theorem~\ref{general theorem: psi depending on (p,q)}.
\end{proof}

Theorem~\ref{general theorem: psi depending on (p,q)} can be applied in various situations beyond what has already been discussed above. For example, the divergence results of \cite{HK13} can be obtained by using Theorem~\ref{general theorem: psi depending on (p,q)} with
$$
\Phi:=\left(\begin{array}{cc}
             I_u & 0 \\
             0 & 0 \\
           \end{array}
\right)
$$
where $I_u$ is the identity matrix. In what follows we shall give applications of Theorem~\ref{general theorem: psi depending on (p,q)} in which the dependence of $\Psi$ on both $\p$ and $\q$ becomes particularly useful. Namely, we shall extend the results of Dani, Laurent and Nogueira \cite{DLN ref} to Hausdorff measures.

First we establish some notation. For any $d \geq 2$ let $P(\Z^d)$ be the set of points $\bv = (v_1,\dots,v_d) \in \Z^d$ such that $\gcd(v_1,\dots,v_d) = 1$. For any subset $\sigma = \{i_1,\dots,i_{\nu}\}$ of $\{1,\dots,d\}$ with $\nu \geq 2$ let $P(\sigma)$ be the set of points $\bv \in \Z^d$ such that $\gcd(v_{i_1},\dots,v_{i_{\nu}}) = 1$. Next, given a partition $\pi$ of $\{1,\dots,d\}$ into disjoint subsets $\pi_\ell$ of at least two elements, let $P(\pi)$ be the set of points $\bv \in \Z^d$ such that $\bv \in P(\pi_\ell)$ for all components $\pi_\ell$ of $\pi$.

Given an approximating function $\psi: \N \to \R^+$ and fixed $\Phi \in \I^{mm}$ and $\y \in \I^m$, let $\cM_{n,m}^{\y,\Phi}(\psi)$ be the set of $\x \in \I^{nm}$ such that
\begin{align} \label{M inequality}
|\q\x + \p\Phi - \y| &< \psi(|\q|)
\end{align}
holds for $(\p,\q) \in \Z^m\times\Z^n \setminus \{\0\}$ with arbitrarily large $|\q|$. Also, given a partition $\pi$ of $\{1,\dots,m+n\}$, let $\cM_{n,m}^{\pi,\y,\Phi}(\psi)$ denote the set of $\x \in \I^{nm}$ for which (\ref{M inequality}) is satisfied for $(\p,\q) \in \Z^m \times \Z^n \setminus \{\0\}$ with arbitrarily large $|\q|$ and with $(q_1,\dots,q_n,p_1,\dots,p_m) \in P(\pi)$.
Now specialising Theorem~\ref{general theorem: psi depending on (p,q)} for the approximating function
\[\Psi(\p,\q) := \left\{
\begin{array}{lcl}
 \displaystyle \psi(|\q|) \quad & \mbox{ \text{if }
 $(q_1,\dots,q_n,p_1,\dots,p_m) \in P(\pi)$,}\\[5ex]
 \displaystyle 0 \quad & \mbox{ \text{otherwise,}}
\end{array}
\right.\]
gives the following.

\begin{theorem} \label{general M statement}
Let $\psi: \N \to \R^+$ be an approximating function such that $\frac{\psi(q)}{q} \to 0$ as $q \to \infty$. Let $\pi$ be any partition of $\{1,\dots,m+n\}$ and let $\Phi \in \I^{mm} \setminus \{\0\}$ and $\y \in \I^m$ be fixed. Let $f$ and $g: r \to g(r):=r^{-m(n-1)}f(r)$ be dimension functions such that $r^{-nm}f(r)$ is monotonic and let $\theta: \N \to \R^+$ be defined by $\theta(q) = q\,g\left(\frac{\psi(q)}{q}\right)^{\frac{1}{m}}$. Then
$$
|\cM_{n,m}^{\pi,\y,\Phi}(\theta)| = 1\qquad\text{implies}\qquad\cH^f(\cM_{n,m}^{\pi,\y,\Phi}(\psi)) = \cH^f(\I^{nm}).
$$
\end{theorem}

\medskip

Now, let us turn our attention to the results of Dani, Laurent and Nogueira from \cite{DLN ref}. For the moment, we will return to the homogeneous setting.
Given a partition $\pi$ of $\{1,\dots,m+n\}$ and an approximating function $\psi: \N \to \R^+$ we will denote by $\cA_{n,m}^{\pi}(\psi)$ the set of $\x \in \I^{nm}$ such that
\[|\q\x+\p|<\psi(|\q|)\]
holds for $(\p,\q) \in \Z^m \times \Z^n \setminus \{\0\}$ with arbitrarily large $|\q|$ and $(q_1, \dots, q_n, p_1, \dots, p_m) \in P(\pi)$. We note that in this case the inequality holds for $(\p,\q) \in \Z^m \times \Z^n \setminus \{\0\}$ with arbitrarily large $|\q|$ if and only if the inequality holds for infinitely many $(\p,\q) \in \Z^m \times \Z^n \setminus \{\0\}$. The notation $\cA_{n,m}(\psi)$ will be used as defined in \S\ref{KG section}.
The following statement is a consequence of \cite[Theorem 1.2]{DLN ref}.

\begin{theorem DLN1*}
Let $n, m \in \N$ and let $\pi$ be a partition of $\{1,\dots,m+n\}$ such that every component of $\pi$ has at least $m+1$ elements. Let $\psi: \N \to \R^+$ be a function such that the mapping $x \to x^{n-1}\psi(x)^m$ is non-increasing. Then,
$$
|\cA_{n,m}^{\pi}(\psi)| =\left\{
\begin{array}{lcl}
 \displaystyle 0 \quad & \mbox{ \text{if }
 $\sum_{q=1}^{\infty}{q^{n-1}\psi(q)^m} < \infty$,}\\[5ex]
 \displaystyle 1 \quad & \mbox{ \text{if }
 $\sum_{q=1}^{\infty}{q^{n-1}\psi(q)^m} = \infty$.}
\end{array}
\right.
$$
\end{theorem DLN1*}


\noindent The following Hausdorff measure analogue of Theorem DLN1 follows from Theorem \ref{general M statement}.

\begin{theorem} \label{DLN Theorem 1.2 Hausdorff}
Let $n,m \in \N$ and let $\pi$ be a partition of $\{1,\dots,m+n\}$ such that every component of $\pi$ has at least $m+1$ elements. Let $\psi: \N \to \R^+$ be an approximating function. Let $f$ and $g: r \to g(r):=r^{-m(n-1)}f(r)$ be dimension functions such that the function $r^{-nm}f(r)$ is monotonic and $q^{n+m-1}g\left(\frac{\psi(q)}{q}\right)$ is non-increasing. Then,
$$
\cH^f(\cA_{n,m}^{\pi}(\psi)) =\left\{
\begin{array}{lcl}
 \displaystyle 0 \quad & \mbox{ \text{if }
 $\sum_{q=1}^{\infty}{q^{n+m-1}g\left(\frac{\psi(q)}{q}\right)} < \infty$,}\\[5ex]
 \displaystyle \cH^f(\mathbb{I}^{nm}) \quad & \mbox{ \text{if }
 $\sum_{q=1}^{\infty}{q^{n+m-1}g\left(\frac{\psi(q)}{q}\right)} = \infty$.}
\end{array}
\right.
$$
\end{theorem}

\begin{proof}
First note that in the light of the fact that $q^{n+m-1}g\left(\frac{\psi(q)}{q}\right)$ is non-increasing we may assume without loss of generality that $\frac{\psi(q)}{q} \to 0$ as $q \to \infty$. To see this, suppose that $\frac{\psi(q)}{q} \nrightarrow 0$. Therefore, there must exist some $\varepsilon > 0$ such that $\frac{\psi(q)}{q} \geq \varepsilon$ infinitely often. In turn, since $g$ is a dimension function, and hence non-decreasing, this means that $q^{n+m-1}g\left(\frac{\psi(q)}{q}\right) \geq q^{n+m-1}g(\varepsilon)$ infinitely often. However, since this expression is non-increasing, we must have that $g(\varepsilon) = 0$. In particular, this means that $g(r) = 0$ and, hence, also $f(r) = 0$ for all $r \leq \varepsilon$. Thus $\cH^f(X) = 0$ for any $X \subset \I^{nm}$ and so the result is trivially true.

In view of the conditions imposed on $\pi$, we must have that $nm>1$. Furthermore, since $\cA_{n,m}^{\pi}(\psi) \subset \cA_{n,m}(\psi)$, it follows from Theorem \ref{Khintchine--Groshev Hausdorff analogue theorem} that $\cH^f(\cA_{n,m}^{\pi}(\psi)) = 0$ when $\sum_{q=1}^{\infty}{q^{n+m-1}g\left(\frac{\psi(q)}{q}\right)} < \infty$. Alternatively, one can use a standard covering argument to obtain a direct proof of the convergence part of Theorem~\ref{DLN Theorem 1.2 Hausdorff}.

Regarding the divergence case, observe that $\cA_{n,m}^{\pi}(\psi) = \cM_{n,m}^{\pi,\0,I_m}(\psi)$, where $I_m$ represents the $m\times m$ identity matrix. Therefore, if $|\cM_{n,m}^{\pi,\0,I_m}(\theta)| = |\cA_{n,m}^{\pi}(\theta)| = 1$ where $\theta: \N \to \R^+$ is defined by $\theta(q) = q\,g\left(\frac{\psi(q)}{q}\right)^{\frac{1}{m}}$, then it would follow from Theorem~\ref{general M statement} that $\cH^f(\cA_{n,m}^{\pi}(\psi)) = \cH^f(\cM_{n,m}^{\pi,\0,I_m}(\psi)) = \cH^f(\I^{nm})$.

Now, by Theorem DLN1, $|\cA_{n,m}^{\pi}(\theta)| = 1$ if $q \to q^{n-1}\theta(q)^m$ is non-increasing and $\sum_{q=1}^{\infty}{q^{n-1}\theta(q)^m} = \infty$. We have that $q^{n-1}\theta(q)^m = q^{n+m-1}g\left(\frac{\psi(q)}{q}\right)$ which is non-increasing by assumption. By our hypotheses, we also have
\[\sum_{q=1}^{\infty}{q^{n-1}\theta(q)^m} = \sum_{q=1}^{\infty}{q^{n+m-1}g\left(\frac{\psi(q)}{q}\right)} = \infty.\]
Hence the proof is complete.
\end{proof}

\medskip

If $\psi(q) := q^{-\tau}$ for some $\tau > 0$ let us write $\cA_{n,m}^{\pi}(\tau) := \cA_{n,m}^{\pi}(\psi)$. The following result regarding the Hausdorff dimension of $\cA_{n,m}^{\pi}(\tau)$ is a corollary of Theorem \ref{DLN Theorem 1.2 Hausdorff}.

\begin{corollary}
Let $n, m \in \N$ and let $\pi$ be a partition of $\{1,\dots,m+n\}$ such that every component of $\pi$ has at least $m + 1$ elements. Then
\[\dimh(\cA_{n,m}^{\pi}(\tau))
       = \left\{
         \begin{array}{ll}
         m(n-1)+\frac{m+n}{\tau + 1} & \text{when } \;\;\; \tau > \frac{n}{m}\;
         ,\\[3ex]
         nm & \text{when } \;\;\; \tau \leq \frac{n}{m}.
         \end{array}\right.\]
\end{corollary}

\begin{proof}
For $\tau \geq \frac{n}{m}$ the result follows on applying Theorem~\ref{DLN Theorem 1.2 Hausdorff} with
$$
f_{\delta}(r) := r^{s_0 + \delta}\qquad\text{where}\qquad s_0 = m(n-1)+\frac{m+n}{\tau + 1}\,.
$$
Indeed, with $\delta$ sufficiently small, all the conditions of Theorem~\ref{DLN Theorem 1.2 Hausdorff} are met and furthermore, as is easily seen, we have from Theorem \ref{DLN Theorem 1.2 Hausdorff} that
$$
\cH^{f_{\delta}}(\cA_{n,m}^{\pi}(\tau)) =\left\{
\begin{array}{lcl}
 \displaystyle 0 \quad & \mbox{ \text{if }
 $\delta > 0$,}\\[3ex]
 \displaystyle \cH^{f_{\delta}}(\mathbb{I}^{nm}) \quad & \mbox{ \text{if }
 $\delta \leq 0$.}
\end{array}
\right.
$$
This means that $\cH^{s_0+\delta}(\cA_{n,m}^{\pi}(\tau)) = 0$ for $\delta > 0$ and $\cH^{s_0+\delta}(\cA_{n,m}^{\pi}(\tau)) = \cH^{s_0+\delta}(\I^{nm})$ for $\delta \leq 0$.
Therefore, if $s_0 \leq nm$ then $\dimh(\cA_{n,m}^{\pi}(\tau)) = s_0$ since, in this case, $\cH^{s_0+\delta}(\I^{nm}) = \infty$ whenever $\delta < 0$. Finally, note that $s_0 \leq nm$ if and only if $\tau \geq \frac{n}{m}$.

In the case where $\tau < \frac{n}{m}$ observe that $\cA_{n,m}^{\pi}(\tau) \supseteq \cA_{n,m}^{\pi}(\frac{n}{m})$ so $\dimh(\cA_{n,m}^{\pi}(\tau)) \geq \dimh(\cA_{n,m}^{\pi}(\frac{n}{m})) = nm$. Combining this with the trivial upper bound gives $\dimh(\cA_{n,m}^{\pi}(\tau)) = nm$ when $\tau < \frac{n}{m}$, as required.
\end{proof}

Next we consider two results of Dani, Laurent and Nogueira regarding inhomogeneous approximation. As before, for a fixed $\y \in \I^m$ we let $\cA_{n,m}^{\y}(\psi)$ denote the set of points $\x \in \I^{nm}$ for which
\begin{align} \label{inhomogeneous inequality}
|\q\x+\p-\y| &< \psi(|\q|)
\end{align}
holds for infinitely many $(\p,\q) \in \Z^m \times \Z^n \setminus \{\0\}$.
Given a partition $\pi$ of $\{1,\dots,m+n\}$, let $\cA_{n,m}^{\pi,\y}(\psi)$ be the set of points $\x \in \I^{nm}$ for which (\ref{inhomogeneous inequality}) holds for infinitely many $(\p,\q) \in \Z^m \times \Z^n \setminus \{\0\}$ with $(q_1,\dots,q_n,p_1,\dots,p_m) \in P(\pi)$.

Rephrasing it in a way which is more suitable for our current purposes, a consequence of \cite[Theorem 1.1]{DLN ref} reads as follows.

\begin{theorem DLN2*}
Let $n, m \in \N$ and let $\pi$ be a partition of $\{1,\dots,m+n\}$ such that every component of $\pi$ has at least $m+1$ elements. Let $\psi: \N \to \R^+$ be a function such that the mapping $x \to x^{n-1}\psi(x)^m$ is non-increasing. Then,
\begin{enumerate}[{\rm(i)}]
\item{if $\sum_{q=1}^{\infty}{q^{n-1}\psi(q)^m} = \infty$ then for almost every $\y \in \I^m$ we have $\left|\cA_{n,m}^{\pi,\y}(\psi)\right|=1$.}
\item{if $\sum_{q=1}^{\infty}{q^{n-1}\psi(q)^m} < \infty$ then for any $\y \in \I^m$ we have $\left|\cA_{n,m}^{\y}(\psi)\right|=0$.}
\end{enumerate}
\end{theorem DLN2*}

The corresponding Hausdorff measure statement we obtain in this case is:

\begin{theorem} \label{DLN Theorem 1.1 Hausdorff}
Let $n,m \in \N$ and let $\pi$ be a partition of $\{1,\dots,m+n\}$ such that every component of $\pi$ has at least $m+1$ elements. Let $\psi: \N \to \R^+$ be an approximating function. Let $f$ and $g: r \to g(r):=r^{-m(n-1)}f(r)$ be dimension functions such that the function $r^{-nm}f(r)$ is monotonic and $q^{n+m-1}g\left(\frac{\psi(q)}{q}\right)$ is non-increasing. Then,
\begin{enumerate}[{\rm(i)}]
\item{if $\sum_{q=1}^{\infty}{q^{n+m-1}g\left(\frac{\psi(q)}{q}\right)}=\infty$ then for Lebesgue almost every $\y \in \I^m$ we have $\cH^f\left(\cA_{n,m}^{\pi,\y}(\psi)\right)=\cH^f(\I^{nm})$.}
\item{if $\sum_{q=1}^{\infty}{q^{n+m-1}g\left(\frac{\psi(q)}{q}\right)}<\infty$ then for any $\y \in \I^m$ we have $\cH^f(\cA_{n,m}^{\y}(\psi))=0$.}
\end{enumerate}
\end{theorem}

\begin{proof}
This is similar to the proof of Theorem~\ref{DLN Theorem 1.2 Hausdorff} with the only difference being the introduction of $\vv y$.
\end{proof}

Finally, let us re-introduce the parameter $\Phi \in \I^{mm}$. In this case, considering the sets $\cM_{n,m}^{\pi,\y,\Phi}(\psi)$ (as defined on page \pageref{M inequality}), it follows from \cite[Theorem 1.3]{DLN ref} that we have:

\begin{theorem DLN3*}
Let $n, m \in \N$ and let $\pi$ be a partition of $\{1,\dots,m+n\}$ such that every component of $\pi$ has at least $m+1$ elements. Let $\psi: \N \to \R^+$ be a function such that the mapping $x \to x^{n-1}\psi(x)^m$ is non-increasing. Then, for any $\y \in \I^m$,
\begin{enumerate}[{\rm(i)}]
\item{if $\sum_{q=1}^{\infty}{q^{n-1}\psi(q)^m} = \infty$ then for almost every $\Phi \in \I^{mm}$ we have that $|\cM_{n,m}^{\pi,\y,\Phi}(\psi)|=1$.}
\item{if $\sum_{q=1}^{\infty}{q^{n-1}\psi(q)^m} < \infty$ then for any $\Phi \in \I^{mm}$ we have $|\cM_{n,m}^{\y,\Phi}(\psi)|=0$.}
\end{enumerate}
\end{theorem DLN3*}


Combining this with Theorem \ref{general M statement} we obtain the following Hausdorff measure statement.

\begin{theorem} \label{DLN Theorem 1.3 Hausdorff}
Let $n, m \in \N$ and let $\pi$ be a partition of $\{1,\dots,m+n\}$ such that every component of $\pi$ has at least $m+1$ elements. Let $\psi: \N \to \R^+$ be an approximating function. Let $f$ and $g: r \to g(r) := r^{-m(n-1)}f(r)$ be dimension functions such that the function $r^{-nm}f(r)$ is monotonic and $q^{n+m-1}g\left(\frac{\psi(q)}{q}\right)$ is non-increasing. Then, for any $\y \in \I^m$,
\begin{enumerate}[{\rm(i)}]
\item{if $\sum_{q=1}^{\infty}{q^{n+m-1}g\left(\frac{\psi(q)}{q}\right)}=\infty$ then for Lebesgue almost every $\Phi \in \I^{mm}$ we have that $\cH^f(\cM_{n,m}^{\pi,\y,\Phi}(\psi)) = \cH^f(\I^{nm})$.}
\item{if $\sum_{q=1}^{\infty}{q^{n+m-1}g\left(\frac{\psi(q)}{q}\right)}<\infty$ then, for any $\Phi \in \I^{mm}$, we have that $\cH^f(\cM_{n,m}^{\y,\Phi}(\psi))=0$.}
\end{enumerate}
\end{theorem}

\begin{proof}
Once again the proof is similar to that of Theorem~\ref{DLN Theorem 1.2 Hausdorff}.
\end{proof}

\section{Preliminaries to the proof of Theorem~\ref{mtp for linear forms theorem}} \label{preliminaries section}

\subsection{Hausdorff measures \label{HM}}

In this section we give a brief account of Hausdorff measures and dimension. Throughout, by a {\em dimension function} $f \, : \, \R^+ \to \R^+ $ we shall mean a left continuous, non-decreasing function such that $f(r)\to 0$ as $r\to 0 \, $. Given
a ball $B:=B(x,r)$ in $\R^k$, we define
\[ V^f(B)\,:=\,f(r)\]
and refer to $V^f(B)$ as the {\em $f$-volume of $B$}.  Note that if $m$ is $k$--dimensional Lebesgue measure and
$f(x)=m(B(0,1))x^k$, then $V^f$ is simply the volume of $B$ in the
usual geometric sense; i.e. $V^f(B)=m(B)$. In the case when
$f(x)=x^s$ for some $s\geq0$, we write $V^s$ for $V^f$.

The Hausdorff $f$--measure with respect to the dimension function
$f$ will be denoted throughout by ${\cal H}^{f}$ and is defined as
follows. Suppose $F$ is  a  subset  of $\R^k$. For $\rho
> 0$, a countable collection $ \left\{B_{i} \right\} $ of balls in
$\R^k$ with radii $r(B_i) \leq \rho $ for each $i$ such that $F \subset
\bigcup_{i} B_{i} $ is called a {\em $ \rho $-cover for $F$}.
Clearly such a cover exists for every $\rho > 0$. For a dimension
function $f$ define
\[ {\cal H}^{f}_{\rho} (F)  :=  \inf \left\{ \sum_{i} V^f(B_i): \left\{B_{i} \right\} \text{is a $\rho$--cover for $F$}\right\}.\]
The {\it Hausdorff $f$--measure}, $ {\cal H}^{f} (F)$, of $F$ with respect to
the dimension function $f$ is defined by
\[ {\cal H}^{f} (F) := \lim_{ \rho \rightarrow 0} {\cal H}^{f}_{\rho} (F) \; = \;
\sup_{\rho > 0 } {\cal H}^{f}_{\rho} (F) \; . \]

A simple consequence of the definition of $ {\cal H}^f $ is the
following useful fact (see, for example, \cite{Falconer ref}).

\begin{lemma} \label{dimfunlemma}
If $ \, f$ and $g$ are two dimension functions such that the ratio
$f(r)/g(r) \to 0 $ as $ r \to 0 $, then ${\cal H}^{f} (F) =0 $
whenever ${\cal H}^{g} (F) < \infty $.
\end{lemma}

In the case that  $f(r) = r^s$ ($s \geq 0$), the measure $ \hf $
is the usual $s$--dimensional Hausdorff measure $\hs $ and the
Hausdorff dimension $\dimh F$ of a set $F$ is defined by
\[ \dimh \,
F \, := \, \inf \left\{ s \geq 0 : {\cal H}^{s} (F) =0 \right\}. \]

For subsets of $\R^k$, $\cH^k$ is comparable to the $k$--dimensional Lebesgue measure. Actually, $\cH^k$ is a constant multiple of the $k$--dimensional Lebesgue measure (but we shall not need this stronger statement).

Furthermore, for any ball $B$ in $\R^k$ we have that $V^k(B)$ is comparable to $|B|$. Thus there are constants $0<c_1<1<c_2<\infty$ such that for any ball $B$ in $\R^k$ we have
\begin{align}\label{volume comparison}
c_1\ V^k(B)\le \cH^k(B)\le c_2\ V^k(B).
\end{align}

A general and classical method for obtaining a lower bound for the Hausdorff $f$-measure of an arbitrary set $F$ is the following mass distribution
principle. This will play a central role in our proof of Theorem \ref{mtp for linear forms theorem} in \S\ref{proof section}.

\begin{lemma}[Mass Distribution Principle] \label{mass distribution principle}
Let $ \mu $ be a probability measure supported on a subset $F$ of $\R^k$.
Suppose there are  positive constants $c$ and $r_o$ such that $$
\mu ( B ) \leq \, c \;  V^f(B) \;
$$ for any ball $B$ with
radius $r \leq r_o \, $.  If $E$ is a subset of $F$ with $\mu(E) =
\lambda > 0$  then $ {\cal H}^{f} (E) \geq \lambda/c \, $.
\end{lemma}

The above lemma is stated as it appears in \cite{BV MTP} since this version is most useful for our current purposes. For further information in general regarding Hausdorff measures and dimension we refer the reader to \cite{Falconer ref,Mattila ref}.

\subsection{The 5$r$--covering lemma}

Let $B:=B(x,r)$ be a ball in $\R^k$. For any $\lambda>0$, we denote by  $\lambda B$ the ball $B$ scaled by a factor $\lambda$; i.e.  $\lambda B(x,r):= B(x, \lambda r)$.

We conclude this section by stating a basic, but extremely useful, covering lemma which we will use throughout \cite{Mattila ref}.

\begin{lemma}[The $5r$--covering lemma]\label{5r}
Every family ${\cal F}$ of balls of uniformly bounded diameter in
$\R^k$ contains a disjoint subfamily ${\cal G}$ such that
\[ \bigcup_{B \in {\cal F} } B \ \subset \ \bigcup_{B
\in {\cal G} } 5B \ \ \ . \]

\end{lemma}

\section{The $K_{G,B} $  covering lemma} \label{KBG lemma section}

Our strategy for proving Theorem \ref{mtp for linear forms theorem} is similar to that used for proving the Mass Transference Principle for balls in \cite{BV MTP}. There are however various technical differences that account for a different shape of approximating sets. First of all we will require a covering lemma analogous to the $\kgb$--lemma established in \cite[Section 4]{BV MTP}. This appears as Lemma \ref{kgb lemma} below. The balls obtained from Lemma \ref{kgb lemma} correspond to planes in the $\limsup$ set $\La(g(\U_j)^{\frac{1}{m}})$. Furthermore, for the proof of Theorem \ref{mtp for linear forms theorem} it is necessary for us to obtain from each of these ``larger" balls a collection of balls which correspond to the ``shrunk" $\limsup$ set $\La(\U)$. The desired properties of this collection and the existence of such a collection are the contents of Lemma \ref{C(A;n) lemma} of this section.

To save on notation, throughout this section let $\tU_j := g(\U_j)^{\frac{1}{m}}$. For an arbitrary ball $B \in \R^k$ and for each $j \in \N$ define
\[ \Pj(B) := \{B(\mathbf{x}, \tU_j) \subset B :\mathbf{x} \in R_j\}\,.\]
Analogously to Lemma 5 from \cite{BV MTP} we will require the following covering lemma.

\begin{lemma} \label{kgb lemma}
Let $\cR$, $\U$, $g$ and $\Omega$ be as in Theorem~\ref{mtp for linear forms theorem} and assume that \eqref{w1} is satisfied. Then for any ball $B$ in $\Omega$ and any $G \in \N$, there exists a finite collection
$$
\kgb \subset \big\{\aj:j \geq G,~A \in \Pj(B)\big\}
$$
satisfying the following properties\/$:$
\begin{enumerate}[{\rm(i)}]
\item{~if $\aj \in \kgb$ then $3A \subset B$;}
\item{~if $\aj, \ajj \in \kgb$ are distinct then $3A \cap 3A' = \emptyset$; and}
\item{~$\cH^k\left(\bigcup\limits_{\aj \in \kgb}{A}\right) \geq \frac{1}{4 \times 15^k}\cH^k(B)$.}
\end{enumerate}
\end{lemma}

\begin{remark}
Essentially, $\kgb$ is a collection of balls drawn from the families $\Pj(B)$. We write $\aj$ for a generic ball from $\kgb$ to `remember' the index $j$ of the family $\Pj(B)$ that the ball $A$ comes from. However, when we are referring only to the ball $A$ (as opposed to the pair $\aj$) we will just write $A$. Keeping track of the associated $j$ will be absolutely necessary in order to be able to choose the `right' collection of balls within $A$ that at the same time lie in an $\U_j$-neighborhood of the relevant $R_j$. Indeed, for $j\neq j'$ we could have $A = A'$ for some $A \in \Pj(B)$ and $A' \in \Phi_{j'}(B)$.
\end{remark}

\begin{proof}[Proof of Lemma~\ref{kgb lemma}]
For each $j \in \N$, consider the collection of balls
$$
\Pj^3(B) := \{B(\mathbf{x},3\tU_j) \subset B: \mathbf{x} \in R_j\}.
$$
By \eqref{w1}, for any $G\ge1$ we have that
$$
\cH^k\left(\bigcup_{j \geq G}{(\Delta(R_j, 3\tU_j) \cap B)}\right) = \cH^k(B).
$$
Observe that
$$
{\bigcup_{L \in \Pj^3(B)}{L}}~~\subset~~ \Delta(R_j, 3\tU_j) \cap B
$$
and that the difference of the two sets lies within $3\tU_j$ of the boundary of $B$. Then, since $\U_j \to 0$, and consequently $\tU_j \to 0$, as $j \to \infty$, we have that
\[
\cH^k\left(\bigcup_{j \geq G}{\bigcup_{L\in \Pj^3(B)}{L}}\right) \sim \cH^k\left(\bigcup_{j \geq G}{(\Delta(R_j, 3\tU_j) \cap B)}\right) = \cH^k(B) \quad \text{as } G \to \infty.
\]
In  particular, there exists a sufficiently large $G' \in \N$ such that for any $G \geq G'$ we have
\[ \cH^k\left(\bigcup_{j \geq G}{\bigcup_{L \in \Pj^3(B)}{L}}\right) \geq \frac{1}{2} \cH^k(B).\]
However, for any $G < G'$ we also have 
\[\bigcup_{j \geq G}{\bigcup_{L \in \Pj^3(B)}{L}} \supset \bigcup_{j \geq G'}{\bigcup_{L \in \Pj^3(B)}{L}}.\]
Thus, for any $G \in \N$ we must have 
\begin{align} \label{kgb lemma Hausdorff}
\cH^k\left(\bigcup_{j \geq G}{\bigcup_{L \in \Pj^3(B)}{L}}\right) &\geq \frac{1}{2} \cH^k(B).
\end{align}
(In fact, using the same argument as above it is possible to show that for any $G \in \N$ we have $\cH^k\left(\bigcup_{j \geq G}{\bigcup_{L \in \Pj^3(B)}{L}}\right) \geq (1-\varepsilon)\cH^k(B)$ for any $0 < \varepsilon < 1$ and hence that we must have $\cH^k\left(\bigcup_{j \geq G}{\bigcup_{L \in \Pj^3(B)}{L}}\right) = \cH^k(B)$. However, (\ref{kgb lemma Hausdorff}) is sufficient for our purposes here.)

By Lemma \ref{5r}, there exists a disjoint subcollection $\cG \subset \{\Lj:j\ge G,~L\in \Pj^3(B)\}$ such that
\[ \bigcup_{(L;j) \in \cG}^{\circ}{L} \subset \bigcup_{j \geq G}{\bigcup_{L \in \Pj^3(B)}{L}} \subset \bigcup_{(L;j) \in \cG}{5L}.\]

Now, let $\cG'$ consist of all the balls from $\cG$ but shrunk by a factor of 3; so the balls in $\cG'$ will still be disjoint when scaled by the factor of 3. Formally,
$$
\cG':=\{(\tfrac13L;j):(L;j)\in\cG\}.
$$
Then, we have that
\begin{align} \label{kgb lemma inclusions}
\bigcup_{\aj \in \cG'}^{\circ}{A} ~\subset~ \bigcup_{j \geq G}~{\bigcup_{L\in \Pj^3(B)}{L}} \subset \bigcup_{\aj \in \cG'}{15A}.
\end{align}
From (\ref{kgb lemma Hausdorff}) and (\ref{kgb lemma inclusions}) we have
\begin{align*}
\cH^k\left(\bigcup_{\aj \in \cG'}{A}\right) &= \sum_{\aj \in \cG'}{\cH^k(A)} \\
                                          &= \sum_{\aj \in \cG'}{\frac{1}{15^k}{\cH^k(15A)}} \\
                                          &\geq \frac{1}{15^k}{\cH^k\left(\bigcup_{\aj \in \cG'}{15A}\right)} \\
                                          &\geq \frac{1}{15^k}\cH^k\left(\bigcup_{j \geq G}{\bigcup_{L \in \Pj^3(B)}{L}}\right) \\
                                          &\geq \frac{1}{2 \times 15^k}\cH^k(B).
\end{align*}
Next note that, since the balls in $\cG'$ are disjoint and contained in $B$ and $\tU_j \to 0$ as $n \to \infty$, we have that
\[ \cH^k\left({\bigcup_{\substack{\aj \in \cG' \\ j \geq N}}A}\right) \to 0 \quad \text{as } N \to \infty.\]
Therefore, there exists a sufficiently large $N_0 \in \N$ such that
\[ \cH^k\left({\bigcup_{\substack{\aj \in \cG' \\ j \geq N_0}}A}\right)  < \frac{1}{4 \times 15^k} \cH^k(B).\]
Thus, taking $\kgb$ to be the subcollection of $\aj \in \cG'$ with $G\le j < N_0$ ensures that $\kgb$ is a finite collection of balls while still satisfying the required properties (i)--(iii).
\end{proof}

\medskip

\begin{lemma} \label{C(A;n) lemma}
Let $\cR$, $\U$, $g$, $\Omega$ and $B$ be as in Lemma~\ref{kgb lemma} and assume that \eqref{w1} is satisfied.
Furthermore, assume that $r^{-k}f(r) \to \infty$ as $r \to 0$.
Let $\kgb$ be as in Lemma~\ref{kgb lemma}. Then, provided that $G$ is sufficiently large, for any $\aj\in\kgb$ there exists a collection $\caj$ of balls satisfying the following properties:
\begin{enumerate}[{\rm(i)}]
\item{~each ball in $\caj$ is of radius $\U_j$ and is centred on $R_j;$}
\item{~if $L \in \caj$ then $3L \subset A;$}
\item{~if $L, M \in \caj$ are distinct then $3L \cap 3M = \emptyset;$}
\item{~$\displaystyle\frac{1}{6^k}\,\cH^k\big(\Delta(R_j,\U_j)\cap \tfrac12A \big)~\le~\cH^k\left(\bigcup_{L \in \caj}{L}\right) ~\leq~ \,\cH^k\big(\Delta(R_j,\U_j)\cap A \big)$; and}
\item there exist some constants $d_1,d_2 > 0$ such that
\begin{align} \label{cardinality of C(A;n)}
d_1\times\left(\frac{g(\U_j)^{\frac{1}{m}}}{\U_j}\right)^l ~\leq~ \# \caj &~\leq~ d_2\times\left(\frac{g(\U_j)^{\frac{1}{m}}}{\U_j}\right)^l.
\end{align}
\end{enumerate}
\end{lemma}

\begin{proof}
First of all note that, by the assumption that $r^{-k}f(r) \to \infty$ as $r \to 0$, we have that
\begin{equation}\label{vb1}
  \frac{\U_j}{\tU_j}\to 0\qquad\text{as }n\to\infty\,.
\end{equation}
In particular we can assume that $G$ is sufficiently large so that
\begin{equation}\label{vb2}
\text{$6\U_j<\tU_j$ \qquad for any $n\ge G$.}
\end{equation}
Let $\vv x_1,\dots, \vv x_t\in R_j\cap\tfrac12A$
be any collection of points such that
\begin{equation}\label{sep}
\|\vv x_i-\vv x_{i'}\|> 6\U_j\qquad\text{if }i\neq i'
\end{equation}
and $t$ is maximal possible. The existence of such a collection follows immediately from the fact that
$R_j\cap\tfrac12A$ is bounded and, by \eqref{sep}, the collection is discrete. Let
$$
\cC\aj:=\{B(\vv x_1,\U_j),\dots,B(\vv x_t,\U_j)\}\,.
$$
Thus, Property (i) is trivially satisfied for this collection $\caj$. Recall that, by construction, $A\in\Phi_j(B)$, which means that the radius of $\tfrac12A$ is $\tfrac12\tU_j$.
If $L\in \cC\aj$, say $L:=B(\vv x_i,\U_j)$, then for any $\vv y\in 3L$ we have that $\|\vv y-\vv x_i\|<3\U_j$ while $\|\vv x_i-\vv x_0\|\le\tfrac12\tU_j$. Then, using \eqref{vb2} and the triangle inequality, we get that
$\|\vv y-\vv x_0\|\le\|\vv y-\vv x_i\|+\|\vv x_i-\vv x_0\|\le 3\U_j+\tfrac12\tU_j<\tU_j$. Hence $3L\subset A$ whence Property~(ii) follows. Further, Property~(iii) follows immediately from condition \eqref{sep}.

By the maximality of the collection $\vv x_1,\dots, \vv x_t$, for any $\vv x\in R_j\cap\tfrac12A$ there exists an $\vv x_i$ from this collection such that $\|\vv x-\vv x_i\|\le 6\U_j$. Hence,
\begin{equation}\label{t1}
\Delta(R_j,\U_j)\cap \tfrac12A~\subset~\bigcup_{L\in\cC\aj}6L\,.
\end{equation}
Hence
\begin{align*}
\cH^k(\Delta(R_j,\U_j)\cap \tfrac12A)&\le\sum_{L\in\cC\aj}\cH^k(6L)\\
&\le\sum_{L\in\cC\aj}6^k\cH^k(L)\\
&=
6^k\cH^k\left(\bigcup^\circ_{L\in\cC\aj}L\right)\,.
\end{align*}
On the other hand, by Property~(ii), we have that
\begin{equation}\label{t1+}
\bigcup^\circ_{L\in\cC\aj}L~\subset~\Delta(R_j,\U_j)\cap A\,,
\end{equation}
which together with the previous inequality establishes Property~(iv).

Finally, Property~(v) is an immediate consequence of Property~(iv) upon noting that
    $$
    \cH^k\big(\Delta(R_j,\U_j)\cap \tfrac12A \big)~\asymp~ \,\cH^k\big(\Delta(R_j,\U_j)\cap A \big)~\asymp~
    \U_j^m\tU_j^l
    $$
    and
    $$
\cH^k\left(\bigcup_{L \in \caj}{L}\right) ~=~ \#\caj\cH^k(L)~\asymp~ \#\caj\,\U_j^k\,,
    $$
    where $l$ is the dimension of $R_j$, $m=k-l$ and $L$ is any ball from $\caj$.
\end{proof}

\begin{remark*} Throughout we use the Vinogradov notation, writing $A \ll B$ if $A \leq cB$ for some positive constant $c$ and $A \gg B$ if $A \geq c'B$ for some positive constant $c'$. If $A \ll B$ and $A \gg B$ we write $A \asymp B$.
\end{remark*}

\section{Proof of Theorem \ref{mtp for linear forms theorem}} \label{proof section}

As with the proof of the Mass Transference Principle given in \cite{BV MTP} and the proof of Theorem BV1 given in \cite{BV Slicing}, we begin by noting that we may assume that $r^{-k}f(r) \to \infty$ as $r \to 0$. To see this we first observe that, by Lemma \ref{dimfunlemma}, if $r^{-k}f(r) \to 0$ as $r \to 0$ we have that $\cH^f(B) = 0$ for any ball $B$ in $\R^k$. Furthermore, since $B \cap \La(\U) \subset B$, the result follows trivially.

Now suppose that $r^{-k}f(r) \to \lambda$ as $r\to 0$ for some $0 < \lambda < \infty$. In this case, $\cH^f$ is comparable to $\cH^k$ and so it would be sufficient to show that $\cH^k(B \cap \La(\U)) = \cH^k(B)$. Since $r^{-k}f(r) \to \lambda$ as $r \to 0$ we have that the ratio $\frac{f(r)}{r^k}$ is bounded between positive constants for sufficiently small $r$. In turn, this implies that, in this case, the ratio of the values $g(\U_i)^{\frac{1}{m}}$ and $\U_i$ is uniformly bounded between positive constants. It then follows from \cite[Lemma 4]{BV Zero-one} that
\[\cH^k\left(B \cap \La\left(g(\U)^{\frac{1}{m}}\right)\right) = \cH^k(B \cap \La(\U)).\]
This together with \eqref{w1} then implies the required result in this case.

Thus, for the rest of the proof we may assume without loss of generality that $r^{-k}f(r) \to \infty$ as $r \to 0$. With this assumption it is a consequence of Lemma \ref{dimfunlemma} that $\cH^f(B_0) = \infty$ for any ball $B_0$ in $\Omega$, which we fix from now on. Therefore, our goal for the rest of the proof is to show that
\[\cH^f(B_0 \cap \La(\U)) = \infty.\]
To this end, for any $\eta > 1$, we will construct a Cantor subset $\K_{\eta}$ of $B_0 \cap \La(\U)$ and a probability measure $\mu$ supported on $\K_{\eta}$ satisfying the condition that for any arbitrary ball $D$ of sufficiently small radius $r(D)$ we have
\begin{align} \label{task}
\mu(D) \ll \frac{V^f(D)}{\eta},
\end{align}
where the implied constant does not depend on $D$ or $\eta$.
By the Mass Distribution Principle (Lemma \ref{mass distribution principle}) and the fact that $\K_{\eta} \subset B_0 \cap \La(\U)$, we would then have that $\cH^f(B_0 \cap \La(\U))\ge\cH^f(\K_{\eta})\gg\eta$ and the proof is finished by taking $\eta$ to be arbitrarily large.

\subsection{The desired properties  of $\K_{\eta}$} \label{properties section}

We will construct the Cantor set $\K_{\eta} := \bigcap_{n=1}^\infty \K(n)$ so that each level $\K(n)$ is a finite union of disjoint closed balls and the levels are nested, that is $\K(n)\supset\K(n+1)$ for $n\geq1$. We will denote the collection of balls constituting level $n$ by $K(n)$.
As with the Cantor set in \cite{BV MTP}, the construction of $\K_{\eta}$ is inductive and each level $\K(n)$ will consist of local levels and sub-levels. So, suppose that the $(n-1)$th level $\K(n-1)$ has been constructed. Then, for every $B \in K(n-1)$ we construct the \itshape $(n,B)$-local level\normalfont, $K(n,B)$, which will consist of balls contained in $B$. The collection of balls $K(n)$ will take the form
\[K(n) := \bigcup_{B \in K(n-1)}{K(n,B)}.\]
Looking even more closely at the construction, each $(n,B)$-local level will consist of \itshape local sub-levels \normalfont and will be of the form
\begin{equation}\label{vb7}
K(n,B) := \bigcup_{i=1}^{l_B}{K(n,B,i)}.
\end{equation}
Here, $K(n,B,i)$ denotes the $i$th local sub-level and $l_B$ is the number of local sub-levels. For $n\ge2$ each local sub-level will be define as the union
\begin{equation}\label{vb3}
K(n,B,i):=
\bigcup_{B'\in\cG(n,B,i)}~~{\bigcup_{\aj\in K_{G',B'}}{\caj}}\,,
\end{equation}
where $B'$ will lie in a suitably chosen collection of balls $\cG(n,B,i)$ within $B$, $K_{G',B'}$ will arise from Lemma~\ref{kgb lemma} and $\caj$ will arise from Lemma~\ref{C(A;n) lemma}. It will be apparent from the construction that the parameter $G'$ becomes arbitrarily large as we construct levels. The set of all pairs $\aj$ that contribute to \eqref{vb3} will be denoted by $\tilde K(n,B,i)$. Thus,
$$
\widetilde K(n,B,i):=\bigcup_{B'\in\cG(n,B,i)}~~K_{G',B'}\, \qquad\text{and}\qquad K(n,B,i) = \bigcup_{\aj \in \widetilde K(n,B,i)}{\caj}.
$$
If additionally we start with $\K(1):=B_0$, then in view of the definition of the sets $\caj$ the inclusion $\K_{\eta}\subset B_0 \cap \La(\U)$ is straightforward. Hence the only real part of the proof will be to show the validity of \eqref{task} for some suitable measure supported on $\K_\eta$. This will require several additional properties which are now stated.

\subsection*{The properties of levels and sub-levels of $\K_\eta$}

\begin{enumerate}

\item[{\bf(P0)}] $K(1)$ consists of one ball, namely $B_0$.

\medskip

\item[{\bf(P1)}] For any $n \geq 2$ and any $B\in K(n-1)$ the balls
$$
\{3L\ :\ L\in K(n,B)\}
$$
are disjoint and contained in $B$.

\medskip

\item[{\bf(P2)}] For any $n\ge 2$, any $B\in K(n-1)$ and any $i\in\{1,\ldots, l_B\}$ the local sub-level $K(n,B,i)$ is a finite union of some collections $\caj$ of balls satisfying Properties (i)--(v) of Lemma~\ref{C(A;n) lemma}, where the balls $3A$ are disjoint and contained in $B$.

\medskip

\item[{\bf(P3)}]
For any $n\ge 2$, $B\in K(n-1)$ and $i\in\{1,\ldots,l_B\}$ we have
$$
\sum_{\aj\in\widetilde K(n,B,i)}V^k(A) \ \ge\ c_3\ V^k(B)  $$ where  $c_3
:= \frac{1}{2^{k+3}\times 5^k \times 15^k}\left(\frac{c_1}{c_2}\right)^2 $ with $c_1$ and $c_2$ as defined in \eqref{volume comparison}.

\medskip

\item[{\bf(P4)}] For any $n\ge 2$, $B\in K(n-1)$, any
$i\in\{1,\ldots,l_B-1\}$ and any $L\in K(n,B,i)$ and $M\in
K(n,B,i+1)$ we have
$$
f(r(M))\le \frac{1}{2}\ f(r(L)) \quad \text{and} \quad g(r(M))\le \frac{1}{2}\ g(r(L)).
$$

\item[{\bf(P5)}] The number of local sub-levels is defined by
$$
l_B:=\left\{
\begin{array}{lcl}
 \displaystyle\left[\frac{c_2\,\eta}{c_3\,\cH^k(B)}\right]+1 & , & \mbox{ if
 $B = B_0 := \K(1)$,}\\[5ex]
 \displaystyle\left[\frac{V^f(B)}{c_3\,V^k(B)}\right]+1 & , & \mbox{ if
 $B\in K(n)$ with $n\ge 2$},
\end{array}
\right.
$$
and satisfies $l_B\ge 2$ for $B\in K(n)$ with $n\ge2$.

\end{enumerate}

Properties (P1) and (P2) are imposed to make sure that the balls in the Cantor construction are sufficiently well separated. On the other hand properties (P3) and (P5) make sure that there are ``enough" balls in each level of the construction of the Cantor set. Property (P4) essentially ensures that all balls involved in the construction of a level of the Cantor set are sufficiently small compared with balls involved in the construction of the previous level. All of the properties (P1)--(P5) will play a crucial role in the measure estimates we obtain in \S\ref{Cantor contruction measure estimates section} and \S\ref{arbitrary ball measure estimates section}.

\subsection{The existence of $\K_{\eta}$ \label{kantor1}}

In this section we show that it is possible to construct a Cantor set with the properties outlined in Section \ref{properties section}.
In what follows we will use the following notation:
$$
K_l(n,B) \ := \ \bigcup_{i=1}^{l}K(n,B,i) \qquad\text{and}\qquad \widetilde K_l(n,B) \ := \ \bigcup_{i=1}^{l}\widetilde K(n,B,i)  \ .
$$

\noindent {\bf Level 1.} The first level is defined by
taking the arbitrary ball $B_0$. Thus, $\K(1) := B_0$ and property
(P0) is trivially satisfied.
We proceed  by induction. Assume that the first $(n-1)$ levels
$\K(1)$, $\K(2)$, $\ldots$ , $\K(n-1)$ have been constructed. We now
construct the $n$'th level $\K(n)$.

\noindent {\bf Level n.} To construct
the $n$'th level we will define local levels $K(n,B)$ for each  $B\in
K(n-1)$. Therefore, from now on we fix some ball $B\in K(n-1)$ and a
sufficiently small constant $\ve:=\ve(B)>0$ which will be
determined later. Recall that each local level $K(n,B)$ will consist of
local sub-levels $K(n,B,i)$ where $1\leq i \leq l_B$ and $l_B $ is given
by Property~(P5).
Let $G \in \N$ be sufficiently large so that Lemmas~\ref{kgb lemma} and \ref{C(A;n) lemma} are applicable. Furthermore suppose that $G$ is large enough so that
\begin{equation}\label{radii comparison}
3\U_j< g(\U_j)^{\frac{1}{m}} \qquad\text{ whenever}\qquad j\geq G,
\end{equation}
\begin{equation}\label{epsilon relation}
\frac{\U_j^k}{f(\U_j)}< \ve \
\frac{r(B)^k}{f(r(B))} \qquad\text{ whenever}\qquad j\geq G,
\end{equation}
and
\begin{equation}\label{card}
\left[\frac{f(\U_j)}{c_3\,\U_j^k}\right] \ \geq \ 1
\qquad\text{ whenever}\qquad j\ge G, \
\end{equation}
where $c_3$ is the constant appearing in property (P3) above.
Note that the existence of $G$ satisfying (\ref{radii comparison})--(\ref{card}) follows from the assumption that $r^{-k}f(r) \to \infty$ as $r \to 0$ and the condition that $\U_j \to 0$ as $j \to \infty$.

\noindent{\bf Sub-level 1.} With $B$ and $G$ as above, let $\kgb$ denote the collection of balls arising from Lemma \ref{kgb lemma}. Define the first sub-level of $K(n,B)$ to be
\[
K(n,B,1) := \bigcup_{\aj \in \kgb}~\caj\,,
\]
thus
$$
\widetilde K(n,B,1)=K_{G,B}\qquad\text{and}\qquad \cG(n,B,1)=\{B\}\,.
$$
By the properties of $\caj$ (Lemma \ref{C(A;n) lemma}), it follows that (P1) is satisfied within this sub-level. From the properties of $\kgb$ (Lemma \ref{kgb lemma}) and Lemma \ref{C(A;n) lemma} it follows that (P2) and (P3) are satisfied for $i = 1$.

\noindent{\bf Higher sub-levels.} To construct higher
sub-levels  we argue by induction. For $l<l_B$,  assume that
the sub-levels $K(n,B,1), \ldots, K(n,B,l)$
satisfying  properties (P1)--(P4) with $l_B$ replaced by $l$ have already been defined. We now construct the next sub-level $K(n,B,l+1)$.

As every  sub-level of the construction has to be well separated
from the previous ones, we first  verify that there is enough
`space'  left over in $B$  once we have removed the  sub-levels
$K(n,B,1), \ldots, K(n,B,l)$ from $B$. More precisely, let
\[A^{(l)} \ := \  \mbox{\small{$\frac{1}{2}$}}B \ \setminus
\bigcup_{L\in K_l(n,B)} \!\!\! 4L \ . \]
We will show that
\begin{align}\label{A measure inequality}
\cH^k\big(A^{(l)} \big)\geq \tfrac{1}{2}\
\cH^k(\mbox{\footnotesize{$\frac{1}{2}$}}B) \ .
\end{align}
First, observe that
\begin{align*}
\cH^k\left(\bigcup_{L \in K_l(n,B)}{4L}\right)
              &\leq \sum_{L \in K_l(n,B)}{\cH^k(4L)} \\
              &\stackrel{(\ref{volume comparison})}\leq 4^kc_2 \sum_{L \in K_l(n,B)}{V^k(L)} \\
              &= 4^kc_2 \sum_{i=1}^{l}{\sum_{L \in K(n,B,i)}{V^k(L)}} \\
              &= 4^kc_2 \sum_{i=1}^{l}~{\sum_{\aj \in \widetilde K(n,B,i)}{\#\caj \times \U_j^k}} \\
              &\stackrel{(\ref{cardinality of C(A;n)})}\leq 4^kc_2d_2 \sum_{i=1}^{l}~{\sum_{\aj \in\widetilde K(n,B,i)}{\left(\frac{g(\U_j)^{\frac{1}{m}}}{\U_j}\right)^l\U_j^k}} \\
              &= 4^kc_2d_2 \sum_{i=1}^{l}~{\sum_{\aj \in\widetilde K(n,B,i)}{g(\U_j)^{\frac{l}{m}}\U_j^m}} \\
              &= 4^kc_2d_2 \sum_{i=1}^{l}~{\sum_{\aj \in\widetilde K(n,B,i)}{g(\U_j)^{\frac{k}{m}} \frac{\U_j^m}{g(\U_j)}}} \\
              &= 4^kc_2d_2 \sum_{i=1}^{l}~{\sum_{\aj \in\widetilde  K(n,B,i)}{g(\U_j)^{\frac{k}{m}}\frac{\U_j^k}{f(\U_j)}}}.
\end{align*}
Hence, by (\ref{epsilon relation}), we get that
\begin{align} \label{space}
\cH^k\left(\bigcup_{L \in K_l(n,B)}{4L}\right)
              &\leq 4^kc_2d_2\ve\frac{r(B)^k}{f(r(B))}\sum_{i=1}^{l}~{\sum_{\aj \in\widetilde K(n,B,i)}{g(\U_j)^{\frac{k}{m}}}} \nonumber \\
              &\leq 4^kc_2d_2\ve\frac{r(B)^k}{f(r(B))}\sum_{i=1}^{l}~{\sum_{\aj \in\widetilde K(n,B,i)}{V^k(A)}} \nonumber \\
              &\stackrel{(\ref{volume comparison})}\leq 4^k\frac{c_2}{c_1}d_2\ve\frac{r(B)^k}{f(r(B))}\sum_{i=1}^{l}~{\sum_{\aj \in\widetilde K(n,B,i)}{\cH^k(A)}} \nonumber \\
              &\stackrel{\textbf{(P2)}}\leq 4^k\frac{c_2}{c_1}d_2\ve\frac{r(B)^k}{f(r(B))}l\cH^k(B) \nonumber \\
              &\leq 4^k\frac{c_2}{c_1}d_2\ve\frac{r(B)^k}{f(r(B))}(l_B-1)\cH^k(B).
\end{align}

If $B = B_0$, set
\[ \ve = \ve(B_0) := \frac{1}{2d_2}\left(\frac{c_1}{c_2}\right)^2\frac{c_3}{2^k4^k}\frac{f(r(B_0))}{\eta}.\]
Otherwise, if $B \neq B_0$, set
\[ \ve = \ve(B) := \ve(B_0) \times \frac{\eta}{f(r(B_0))} = \frac{1}{2d_2}\left(\frac{c_1}{c_2}\right)^2\frac{c_3}{2^k4^k}.\]

Then, it follows from (\ref{space}) combined with \textbf{(P5)} that
\[\cH^k\left(\bigcup_{L \in K_l(n,B)}{4L}\right) \leq \tfrac{1}{2}\cH^k\left(\tfrac{1}{2}B\right),\]
thus verifying (\ref{A measure inequality}).

By construction, $K_l(n,B)$ is a finite collection of balls. Therefore, the quantity
$$
d_{\min} := \min \{\diam(L): L \in K_l(n,B)\}
$$
is well-defined and positive.
Let ${\cal A}(n,B,l)$ be the collection of all the balls of diameter $d_{\min}$ centred at a point in
$A^{(l)}$. By the $5r$--covering lemma (Lemma~\ref{5r}), there exists a disjoint subcollection ${\cal G}(n,B,l+1)$ of $\cA(n,B,l)$
such that $$ A^{(l)} \ \subset \ \bigcup_{B' \in \cA(n,B,l)} B' \ \subset \ \bigcup_{B' \in {\cal G}(n,B,l+1)} 5 B'  .
$$
The  collection ${\cal G}(n,B,l+1)$ is clearly
contained within $B$  and, since the balls in this collection are disjoint and of the same size,  it  is finite. Moreover, by construction
\begin{equation}
 B' \cap \bigcup_{L \in K_l(n,B)} 3 L =
\emptyset \hspace{7mm} {\rm for\ any \ }  B'\in {\cal
G}(n,B,l+1) \ ; \label{blcoll} \end{equation} i.e. the balls in
${\cal G}(n,B,l+1) $ do not intersect any of the $3L$ balls from the
previous sub-levels. It follows that
$$
{\cal H}^k \left(\bigcup_{B' \in {\cal G}(n,B,l+1) } 5 B' \right) \ \geq \ {\cal
H}^k (A^{(l)} ) \ \stackrel{(\ref{A measure inequality})}{\ \geq \ } \mbox{
\small $\frac12$}  \ \cH^k(\mbox{\footnotesize{$\frac{1}{2}$}}B) \
. $$
On the other hand, since ${\cal G}(n,B,l+1)$ is a disjoint
collection of balls we have that 
\begin{align*}
{\cal H}^k \left( \bigcup_{B' \in {\cal G}(n,B,l+1) } 5 B' \right) &\leq \sum_{B' \in {\cal G}(n,B,l+1)}{\cH^k(5 B')} \\[3ex]
                                                                   &\stackrel{(\ref{volume comparison})}{\ \leq \
} 5^k \frac{c_2}{c_1} \sum_{{B'\in {\cal G}(n,B,l+1)}}{\cH^k(B')} \\[3ex]
                                                                   &= 5^k \frac{c_2}{c_1} {\cal H}^k\left(\bigcup^\circ_{B' \in {\cal G}(n,B,l+1)}  B' \right) \ .
\end{align*} 
Hence,
\begin{equation}\label{bl Hausdorff}
{\cal H}^k \left(\bigcup^\circ_{B' \in {\cal G}(n,B,l+1) } B'
\right)\,\ge\,\frac{c_1}{2 c_2 5^k} \ \
\cH^k(\mbox{\footnotesize{$\frac{1}{2}$}}B)\,.
\end{equation}

Now we are ready to construct the $(l+1)$th sub-level $K(n,B,l+1)$. Let $G' \ge G+1$ be sufficiently large so that Lemmas~\ref{kgb lemma} and \ref{C(A;n) lemma} are applicable to every ball $B'\in\cG(n,B,l+1)$ with $G'$ in place of $G$. Furthermore, ensure that $G'$ is sufficiently large so that for every $i \geq G'$,
\begin{align} \label{f and g relations}
f(\U_i) \leq \tfrac{1}{2}\min_{L \in K_l(n,B)}{f(r(L))} \qquad \text{and} \qquad g(\U_i) \leq \tfrac{1}{2}\min_{L \in K_l(n,B)}{g(r(L))}.
\end{align}
Imposing the above assumptions on $G'$ is possible since there are only finitely many balls in $K_l(n,B)$, and since $\U_i \to 0$ as $i \to \infty$ and $f$ and $g$ are dimension functions.

Now, to each ball $B'\in \cG(n,B,l+1)$ we apply Lemma \ref{kgb lemma} to obtain a collection of balls $K_{G',B'}$ and define
\[ K(n,B,l+1) := \bigcup_{B'\in \cG(n,B,l+1)}~{\bigcup_{\aj \in K_{G',B'}}}~\caj.\]
Consequently,
$$
\widetilde K(n,B,l+1)=\bigcup_{B'\in\cG(n,B,l+1)}~~K_{G',B'}\,.
$$
Since $G' \geq G$, properties (\ref{radii comparison})--(\ref{card}) remain valid. We now verify properties \textbf{(P1)}--\textbf{(P5)} for this sub-level.

Regarding \textbf{(P1)}, we first observe that it is satisfied for balls in $\bigcup_{\aj \in K_{G',B'}}{\bigcup_{L \in \caj}{L}}$ by the properties of $\caj$ and the fact that the balls in $K_{G',B'}$ are disjoint. Next, since any balls in $K_{G',B'}$ are contained in $B'$ and the balls $B' \in \cG(n,B,l+1)$ are disjoint, it follows that \textbf{(P1)} is satisfied for balls $L$ in $K(n,B,l+1)$. Finally, combining this with (\ref{blcoll}), we see that \textbf{(P1)} is satisfied for balls $L$ in $K_{l+1}(n,B)$. That \textbf{(P2)} is satisfied for this sub-level is a consequence of Lemma \ref{kgb lemma} (i) and (ii) and the fact that the balls $B' \in \cG(n,B,l+1)$ are disjoint.

To establish \textbf{(P3)} for $i = l+1$ note that
\begin{align*}
\sum_{\aj\in \widetilde K(n,B,l+1)}{V^k(A)}
          &= \sum_{B' \in \cG(n,B,l+1)}~{\sum_{\aj \in K_{G',B'}}{V^k(A)}} \\
          &\stackrel{(\ref{volume comparison})}\geq \frac{1}{c_2}\sum_{B' \in \cG(n,B,l+1)}~{\sum_{\aj \in K_{G',B'}}{\cH^k(A)}}. \\
\end{align*}
Then, by Lemma \ref{kgb lemma} and the disjointness of the balls in $\cG(n,B,l+1)$, we have that
\begin{align*}
\sum_{\aj\in \widetilde K(n,B,l+1)}{V^k(A)}
          &\geq \frac{1}{c_2}\sum_{B' \in \cG(n,B,l+1)}{\frac{1}{4 \times 15^k}\cH^k\left(B'\right)} \\
          &= \frac{1}{c_2 \times 4 \times 15^k} \cH^k\left(\bigcup_{B' \in \cG(n,B,l+1)}{B'}\right) \\
          &\stackrel{(\ref{bl Hausdorff})}\geq \frac{1}{c_2 \times 4 \times 15^k}\frac{c_1}{2 \times c_2 \times 5^k}\cH^k\left(\tfrac{1}{2}B\right) \\
          &\stackrel{(\ref{volume comparison})}\geq \frac{1}{2^{k+3}\times 5^k \times 15^k}\left(\frac{c_1}{c_2}\right)^2 V^k(B) \\
          &= c_3 V^k(B).
\end{align*}
Finally, \textbf{(P4)} is trivially satisfied as a consequence of the imposed condition (\ref{f and g relations}) and \textbf{(P5)}, that $l_L \geq 2$ for any ball $L$ in $K(n,B,l+1)$, follows from (\ref{card}).

Hence, properties \textbf{(P1)}--\textbf{(P5)} are satisfied up to the local sub-level $K(n,B,l+1)$ thus establishing the existence of the local level $K(n,B) = K_{l_B}(n,B)$ for each $B \in K(n-1)$. In turn, this establishes the existence of the $n$th level $K(n)$ (and also $\K(n)$).

\subsection{The measure $\mu$ on $\K_\eta$}

In this section, we  define a probability measure $\mu$ supported
on $\K_\eta$. We will eventually show that the measure satisfies
(\ref{task}). For any ball $L \in K(n)$, we attach a weight
$\mu(L)$ defined recursively as follows.

For $n=1$, we have that $L = B_0 := \K(1)$ and we set $\mu(L):=1$. For subsequent levels the measure is defined inductively.

Let $n \geq 2$ and suppose that $\mu(B)$ is defined for every $B\in K(n-1)$. In particular, we have that
$$
\sum_{B\in K(n-1)}\mu(B)=1\,.
$$
Let $L$ be a ball in $K(n)$. By construction, there is a unique ball $B\in K(n-1)$ such that $L\subset B$. Recall, by \eqref{vb7} and \eqref{vb3}, that
$$
K(n,B) := {\bigcup_{\aj\in \widetilde K_{l_B}(n,B)}{\caj}}\,
$$
and so $L$ is an element of one of the collections $\cC\ajj$ appearing in the right hand side of the above. We therefore define
$$
\mu(L):=\frac{1}{\#\cC\ajj}\times \frac{g(\U_{j'})^{\frac{k}{m}}}{\sum\limits_{\aj\in \widetilde K_{l_B}(n,B)}{g(\U_{j})^{\frac{k}{m}}}}\times\mu(B)\,.
$$
Thus $\mu$ is inductively defined on any
ball appearing  in the construction of $\K_{\eta}$. Furthermore, $\mu$ can be uniquely extended in a standard way to all Borel subsets $F$ of $\R^k$ to give a probability measure $\mu$ supported on
$\K_{\eta}$. Indeed, for any Borel subset $F$ of $\R^k$,
\[
\mu(F):= \mu(F \cap  \K_{\eta})  \; = \; \inf\;\sum_{L\in{\cal
C}(F)}\mu(L)  \ ,
\]
where the infimum is taken over all covers ${\cal C}(F)$ of $F
\cap \K_{\eta}$ by balls  $L \in \bigcup\limits_{n\in\N}K(n)$.
See \cite[Proposition 1.7]{Falconer ref} for further details.

We end this section by observing that
\begin{align} \label{initial L measure estimate}
\mu(L)
       &\leq \frac{1}{d_1\left(\frac{g(\U_{j'})^\frac{1}{m}}{\U_{j'}}\right)^l}\times \frac{g(\U_{j'})^{\frac{k}{m}}}{\sum\limits_{\aj\in \widetilde K_{l_B}(n,B)}{g(\U_{j})^{\frac{k}{m}}}}\times\mu(B) \nonumber \\[3ex]
       &= \frac{f(\U_{j'})}{d_1\sum\limits_{\aj\in \widetilde K_{l_B}(n,B)}{g(\U_{j})^{\frac{k}{m}}}}\times\mu(B).
\end{align}
This is a consequence of (\ref{cardinality of C(A;n)}) and the relationship between $f$ and $g$.
In fact the above inequality can be reversed if $d_1$ is replaced  by $d_2$.

\subsection{The measure of a  ball in  the Cantor set construction} \label{Cantor contruction measure estimates section}

The goal of this section is to prove that
\begin{align} \label{measure of balls in Cantor construction}
\mu(L) \ll \frac{V^f(L)}{\eta}
\end{align}
for any ball $L$ in $K(n)$ with $n \geq 2$.
We will begin with the level $n=2$. Fix any ball $L \in K(2) = K(2,B_0)$. Further let $\ajj\in\widetilde K_{l_{B_0}}(2,B_0)$ be such that
$L\in \cC\ajj$. Then, by (\ref{initial L measure estimate}), the definition of $\mu$ and the fact that $\mu(B_0)=1$, we have that
\begin{align} \label{level 2 measure 1}
\mu(L) &\leq \frac{f(\U_{j'})}{d_1\sum\limits_{\aj\in \widetilde K_{l_{B_0}}(2,B_0)}{g(\U_{j})^{\frac{k}{m}}}}~~.
\end{align}
Next, by Properties \textbf{(P3)} and \textbf{(P5)} of the Cantor set construction, we get that
\begin{align} \label{level 2 measure 2}
\sum\limits_{\aj\in \widetilde K_{l_{B_0}}(2,B_0)}{g(\U_{j})^{\frac{k}{m}}}
         &= \sum\limits_{\aj\in \widetilde K_{l_{B_0}}(2,B_0)}{V^k(A)} \nonumber \\
         &=\sum_{i=1}^{l_{B_0}}{\sum_{\aj\in \widetilde K(2,B_0,i)}{V^k(A)}} \nonumber \\
         &\stackrel{\textbf{(P3)}}\geq \sum_{i=1}^{l_{B_0}}{c_3V^k(B_0)} \nonumber \\
         &= l_{B_0}c_3V^k(B_0) \nonumber \\
         &\stackrel{(\ref{volume comparison})}\geq l_{B_0}\frac{c_3}{c_2}\cH^k(B_0) \nonumber \\
         &\stackrel{\textbf{(P5)}}\geq \frac{c_2\eta}{c_3\cH^k(B_0)}\frac{c_3}{c_2}\cH^k(B_0) ~=~ \eta.
\end{align}
Combining (\ref{level 2 measure 1}) and (\ref{level 2 measure 2}) gives \eqref{measure of balls in Cantor construction} as required since $f(\U_{j'})=f(r(L))=V^f(L)$.

Now let $n>2$ and assume that (\ref{measure of balls in Cantor construction}) holds for balls in $K(n-1)$. Consider an arbitrary ball $L$ in  $K(n)$. Then there exists a unique ball $B\in K(n-1)$ such that $L\in K(n,B)$.
Further let $\ajj\in\widetilde K_{l_{B}}(n,B)$ be such that $L\in \cC\ajj$.
Then it follows from (\ref{initial L measure estimate}) and our induction hypothesis that
\begin{align}
\mu(L) &\ll \frac{f(\U_{j'})}{d_1\sum\limits_{\aj\in \widetilde K_{l_B}(n,B)}{g(\U_{j})^{\frac{k}{m}}}}\times \frac{V^f(B)}{\eta}.\label{vb8}
\end{align}
Now, we have that
\begin{align}
\sum\limits_{\aj\in \widetilde K_{l_B}(n,B)}{g(\U_{j})^{\frac{k}{m}}}
       &= \sum_{i=1}^{l_{B}}{\sum_{\aj\in \widetilde K(n,B,i)}{V^k(A)}}\nonumber \\
       &\stackrel{\textbf{(P3)}}\geq \sum_{i=1}^{l_{B}}{c_3V^k(B)} \nonumber \\
       &= l_Bc_3V^k(B) \nonumber \\
       &\stackrel{\textbf{(P5)}}\geq \frac{V^f(B)}{c_3V^k(B)}c_3V^k(B) \nonumber \\
       &=V^f(B).\label{vb9}
\end{align}
Since $V^f(L)=f(\U_j)$, combining \eqref{vb8} and \eqref{vb9} gives (\ref{measure of balls in Cantor construction}) and thus completes the proof of this section.

\subsection{The measure of an  arbitrary ball} \label{arbitrary ball measure estimates section}

Set $r_0 := \min\{r(B): B \in K(2)\}$. Take an arbitrary ball $D$ in $\Om$ such that $r(D) < r_0$. We wish to establish (\ref{task}) for $D$, i.e. we wish to show that
\[\mu(D) \ll \frac{V^f(D)}{\eta}\,,\]
where the implied constant is independent of $D$ and $\eta$.
In accomplishing this goal the following Lemma from \cite{BV MTP} will be useful.

\begin{lemma}\label{separation lemma}
Let $A:=B(x_A,r_A)$ and $M:=B(x_M,r_M)$ be arbitrary balls such that
$A\cap M\not=\emptyset$ and $A\setminus(cM)\not=\emptyset$ for
some $c\ge3$. Then $r_M\,\le \,r_A$ and $cM\subset 5A$.
\end{lemma}

A good part of the subsequent argument will follow the same reasoning as given in \cite[Section 5.5]{BV MTP}. However, there will also be obvious alterations to the proofs that arise from the different construction of a Cantor set used here. Recall that the measure $\mu$ is supported on $\K_{\eta}$. Without loss of generality, we will make the following two assumptions:
\begin{itemize}
\item{$D \cap \K_{\eta} \neq \emptyset$;}
\item{for every $n$ large enough $D$ intersects at least two balls in $K(n)$.}
\end{itemize}

If the first of these were false then we would have $\mu(D)=0$ as $\mu$ is supported on $\K_\eta$ and so (\ref{task}) would trivially follow. If the second assumption were false then $D$ would have to intersect exactly one ball, say $L_{n_i}$ from levels $K_{n_i}$ with arbitrarily large $n_i$. Then, by \eqref{measure of balls in Cantor construction}, we would have $\mu(D) \leq \mu(L_{n_i})\to0$ as $i\to\infty$ and so, again, (\ref{task}) would be trivially true.

By the above two assumptions, we have that there exists a maximum integer $n$ such that
\begin{align}\label{number of balls intersected}
\text{$D$ intersects at least 2 balls from $K(n)$}
\end{align}
and
\begin{align*}
\text{$D$ intersects only one ball $B$ from $K(n-1)$}.
\end{align*}

By our choice of $r_0$, we have that $n > 2$. If $B$ is the only ball from $K(n-1)$ which has non-empty intersection with $D$, we may also assume that $r(D) < r(B)$. To see this, suppose the contrary that $r(B) \leq r(D)$. Then, since $D \cap \K_\eta \subset B$ and $f$ is increasing, upon recalling (\ref{measure of balls in Cantor construction}) we would have
\[\mu(D) \leq \mu(B) \ll \frac{V^f(B)}{\eta} = \frac{f(r(B))}{\eta} \leq \frac{f(r(D))}{\eta} = \frac{V^f(D)}{\eta}, \]
and so we would be done.

Now, since $K(n,B)$ is a cover for $D \cap \K_{\eta}$, we have
\begin{align} \label{measure sum}
\mu(D) &\leq \sum_{i=1}^{l_B}{\sum_{L \in K(n,B,i): L \cap D \neq \emptyset}{\mu(L)}} \nonumber \\
       &= \sum_{i=1}^{l_B}{\sum_{\aj \in\widetilde K(n,B,i)}~{\sum_{\substack{L \in \caj \\ L \cap D \neq \emptyset}}{\mu(L)}}} \nonumber \\
       &\stackrel{(\ref{measure of balls in Cantor construction})}{\ll}~ \sum_{i=1}^{l_B}{\sum_{\aj \in \widetilde K(n,B,i)}~{\sum_{\substack{L \in \caj \\ L \cap D \neq \emptyset}}{\frac{V^f(L)}{\eta}}}}.
\end{align}
To estimate the right-hand side of (\ref{measure sum}) we consider the following types of sub-levels:

\bigskip

\noindent \underline{{\em Case 1}\/} :~ Sub-levels $K(n,B,i)$ for which
\[\#\{L \in K(n,B,i): L \cap D \neq \emptyset\} = 1.\]

\bigskip

\noindent \underline{{\em Case 2}\/} :~ Sub-levels $K(n,B,i)$ for which
\[\#\{L \in K(n,B,i): L \cap D \neq \emptyset\} \geq 2\quad\text{and}\]
\begin{equation}\label{vb10}
\#\{\aj\in\widetilde K(n,B,i) \text{ with } D \cap L\neq \emptyset~\text{for some }L\in\caj\} \geq 2.
\end{equation}

\bigskip

\noindent \underline{{\em Case 3}\/} :~ Sub-levels $K(n,B,i)$ for which
\[\#\{L \in K(n,B,i): L \cap D \neq \emptyset\} \geq 2\quad\text{and}\]
\[\#\{\aj\in\widetilde K(n,B,i) \text{ with } D \cap L\neq \emptyset~\text{for some }L\in\caj\} = 1.\]

Strictly speaking we also need to consider the sub-levels $K(n,B,i)$ for which $\#\{L \in K(n,B,i): L \cap D \neq \emptyset\} = 0$. However, these sub-levels do not contribute anything to the sum on the right-hand side of (\ref{measure sum}).

\medskip

\noindent {\it Dealing with  Case 1}.
Let $K(n,B,i^*)$ denote the first sub-level within Case 1 which has non-empty intersection with $D$. Then there exists a unique ball $L^*$ in $K(n,B,i^*)$ such that $L^* \cap D \neq \emptyset$.
By (\ref{number of balls intersected}) there is another ball $M \in K(n,B)$ such that $M \cap D \neq \emptyset$. By property \textbf{(P1)}, $3L^*$ and $3M$ are disjoint. It follows that $D \setminus 3L^* \neq \emptyset$. Therefore, by Lemma \ref{separation lemma}, we have that $r(L^*) \leq r(D)$ and so, since $f$ is increasing,
\begin{align} \label{case 1 volume comparison}
V^f(L^*) \leq V^f(D).
\end{align}
By Property \textbf{(P4)} we have that for any $i\in\{i^*+1,\ldots,l_B\}$ and any $L\in
K(n,B,i)$ we have that
$$
V^f(L)=f(r(L))\le 2^{-(i-i^*)}\ f(r(L^*))=2^{-(i-i^*)}\ V^f(L^*).
$$
Using these inequalities and (\ref{case 1 volume comparison}) we see that the contribution to the right-hand side of (\ref{measure sum}) from Case 1 is:
\begin{align} \label{case 1}
\sum_{i \in \text{Case 1}}~~{\sum_{\substack{L \in K(n,B,i)\\ L \cap D \neq \emptyset}}{\frac{V^f(L)}{\eta}}}
      &\leq \sum_{i\ge i^*}~2^{-(i-i^*)}{\frac{V^f(L^*)}{\eta}} \le 2{\frac{V^f(L^*)}{\eta}}\le2\frac{V^f(D)}{\eta}.
\end{align}

\medskip

\noindent {\it Dealing with  Case 2}.
Let $K(n,B,i)$ be any sublevel subject to the conditions of Case 2. Then there exist distinct balls
$\aj$ and $(A';j')$ in $\widetilde K(n,B,i)$ and balls $L\in\caj$ and $L'\in\cC(A';j')$ such that $L \cap D \neq \emptyset$ and $L' \cap D \neq \emptyset$.
Since $L \cap D \neq \emptyset$ and $L\subset A$ we have that $A\cap D\neq\emptyset$. Similarly, $A'\cap D\neq\emptyset$. Furthermore, by Property~\textbf{(P2)}, the balls $3A$ and $3A'$ are disjoint and contained in $B$. Hence, $D \setminus 3A \neq \emptyset$. Therefore, by Lemma \ref{separation lemma}, $r(A) \leq r(D)$ and $A\subset 3A \subset 5D$. Hence, on using (\ref{cardinality of C(A;n)}) we get that the contribution to the right-hand side of (\ref{measure sum}) from Case 2 is estimated as follows
\begin{align*}
\sum_{i \in \text{Case 2}}~{\sum_{\aj\in \widetilde K(n,B,i)}~{\sum_{\substack{L \in \caj \\ L \cap D \neq \emptyset}}{\frac{V^f(L)}{\eta}}}}
     &\leq \sum_{i \in \text{Case 2}}~~{\sum_{\substack{\aj\in\widetilde K(n,B,i) \\ A \subset 5D}}{\#\caj\frac{f(\U_j)}{\eta}}} \\
     &\stackrel{(\ref{cardinality of C(A;n)})}\ll \sum_{i \in \text{Case 2}}~{\sum_{\substack{\aj\in\widetilde K(n,B,i) \\ A \subset 5D}}{\left(\frac{g(\U_j)^{\frac{1}{m}}}{\U_j}\right)^l\frac{f(\U_j)}{\eta}}} \\
     &= \sum_{i \in \text{Case 2}}~{\sum_{\substack{\aj\in\widetilde K(n,B,i) \\ A \subset 5D}}{\frac{g(\U_j)^{\frac{l}{m}}\U_j^{-l}\U_j^{l}g(\U_j)}{\eta}}} \\
     &= \sum_{i \in \text{Case 2}}~{\sum_{\substack{\aj\in\widetilde K(n,B,i) \\ A \subset 5D}}{\frac{g(\U_j)^{\frac{l}{m}+1}}{\eta}}} \\
     &= \sum_{i \in \text{Case 2}}~{\sum_{\substack{\aj\in\widetilde K(n,B,i) \\ A \subset 5D}}{\frac{g(\U_j)^{\frac{k}{m}}}{\eta}}} \\
     &= \sum_{i \in \text{Case 2}}~{\sum_{\substack{\aj\in\widetilde K(n,B,i) \\ A \subset 5D}}{\frac{V^k(A)}{\eta}}}.
\end{align*}
Combining this with properties \textbf{(P2)} and \textbf{(P5)} we get
\begin{align*}
\sum_{i \in \text{Case 2}}~{\sum_{\aj\in \widetilde K(n,B,i)}~{\sum_{\substack{L \in \caj \\ L \cap D \neq \emptyset}}{\frac{V^f(L)}{\eta}}}}
     &\stackrel{(\ref{volume comparison})}\ll \frac{1}{c_1\eta}\sum_{i \in \text{Case 2}}~{\sum_{\substack{\aj\in\widetilde K(n,B,i) \\ A\subset 5D}}{\cH^k(A)}} \\
     &\stackrel{\textbf{(P2)}}= \frac{1}{c_1\eta}\sum_{i \in \text{Case 2}}~{\cH^k\left(\bigcup_{\substack{\aj\in\widetilde K(n,B,i) \\ A\subset 5D}}A\right)} \\
     &\leq \frac{1}{c_1\eta}\sum_{i \in \text{Case 2}}{\cH^k(5D)} \\
     &\leq \frac{1}{c_1\eta}5^kl_B\cH^k(D) \\
     &\stackrel{(\ref{volume comparison})}\leq \frac{c_2}{c_1\eta}5^kl_BV^k(D) \\
     &\stackrel{\textbf{(P5)}}\leq \frac{c_2}{c_1\eta}5^k\left(\frac{2V^f(B)}{c_3V^k(B)}\right)V^k(D).
\end{align*}
Recalling our assumption that $r(D) < r(B)$ and the fact that $r^{-k}f(r)$ is decreasing, we obtain that
\begin{align} \label{case 2}
\sum_{i \in \text{Case 2}}~{\sum_{\aj\in\widetilde K(n,B,i)}~{\sum_{\substack{L \in \caj \\ L \cap D \neq \emptyset}}{\frac{V^f(L)}{\eta}}}}
     &\ll \frac{c_2}{c_1\eta}5^k\frac{2}{c_3}\frac{V^f(D)}{V^k(D)}V^k(D) \nonumber \\
     &= \frac{2 c_2 5^k}{c_1c_3}\frac{V^f(D)}{\eta} \nonumber \\
     &\ll \frac{V^f(D)}{\eta}.
\end{align}

\noindent {\it Dealing with  Case 3}.
First of all note that for each level $i$ of Case 3 there exists a unique $(A_i;j_i)\in\widetilde K(n,B,i)$ such that $D$ has a non-empty intersection with balls in $\cC(A_i;j_i)$. Let $K(n,B,i^{**})$ denote the first sub-level within Case 3. Then there exists a ball $L^{**}$ in $K(n,B,i^{**})$ such that $L^{**} \cap D \neq \emptyset$.
By (\ref{number of balls intersected}) there is another ball $M \in K(n,B)$ such that $M \cap D \neq \emptyset$. By property \textbf{(P1)}, $3L^{**}$ and $3M$ are disjoint. It follows that $D \setminus 3L^{**} \neq \emptyset$ and therefore, by Lemma \ref{separation lemma}, we have that $r(L^{**}) \leq r(D)$ and so, since $g$ is increasing, we have that
\begin{align} \label{case 1 volume comparison+}
g(r(L^{**})) \leq g(r(D)).
\end{align}
Furthermore, by Property \textbf{(P4)}, for any $i\in\{i^{**}+1,\ldots,l_B\}$ and any $L\in
K(n,B,i)$ we have that
$$
g(r(L))\le 2^{-(i-i^{**})}\ g(r(L^{**})).
$$
Then, the contribution to the sum (\ref{measure sum}) from Case 3 is estimated as follows
\begin{align}
\sum_{i \in \text{Case 3}}~{\sum_{\aj\in\widetilde K(n,B,i)}~{\sum_{\substack{L \in \caj \\ L \cap D \neq \emptyset}}{\frac{V^f(L)}{\eta}}}}
     &\leq \sum_{i \in \text{Case 3}}~{\sum_{\substack{L \in \cC(A_i;j_i) \\ L \cap D \neq \emptyset}}{\frac{V^f(L)}{\eta}}} \nonumber \\
     &= \sum_{i \in \text{Case 3}}~{\sum_{\substack{L \in \cC(A_i;j_i) \\ L \cap D \neq \emptyset}}{\frac{f(\U_{j_i})}{\eta}}} \nonumber \\
     &\ll \sum_{i \in \text{Case 3}}~{\#\cC(A_i;j_i)\times\frac{f(\U_{j_i})}{\eta}} \nonumber \\
     &\ll \sum_{i \in \text{Case 3}}~{\left(\frac{r(D)}{\U_{j_i}}\right)^l\frac{f(\U_{j_i})}{\eta}} \nonumber \\
     &= \sum_{i \in \text{Case 3}}{r(D)^l\frac{g(\U_{j_i})}{\eta}}\nonumber\\
     &\ll \frac{r(D)^l}{\eta}\sum_{i \in \text{Case 3}}{\frac{g(\U_{j_{i^{**}}})}{2^{i-i^{**}}}} \nonumber \\
     &\leq 2\frac{r(D)^l}{\eta}g(\U_{j_{i^{**}}}). \nonumber
\end{align}
Noting that $\U_{j_i} = r(L^{**})$ and recalling (\ref{case 1 volume comparison+}) we see that
\begin{align}
\sum_{i \in \text{Case 3}}~{\sum_{\aj\in\widetilde K(n,B,i)}~{\sum_{\substack{L \in \caj \\ L \cap D \neq \emptyset}}{\frac{V^f(L)}{\eta}}}} &\leq 2\frac{r(D)^l}{\eta}g(r(D)) = 2\frac{f(r(D))}{\eta} \ll \frac{V^f(D)}{\eta}\label{case 3}.
\end{align}
Finally, combining (\ref{case 1}), (\ref{case 2}) and (\ref{case 3}) together with (\ref{measure sum}) gives $\mu(D) \ll \frac{V^f(D)}{\eta}$ and thus completes the proof of Theorem \ref{mtp for linear forms theorem}.

\end{document}